\documentclass[oneside, 11pt]{amsart} 

\usepackage{amsfonts}
\usepackage{amssymb}
\usepackage{latexsym}
\usepackage{graphics}
\usepackage[2emode]{psfrag}
\usepackage{amsthm}
\usepackage{amsmath}
\usepackage[all]{xy}

\begin{document}  

\newcommand{\norm}[1]{\| #1 \|}
\def\N{\mathbb N}
\def\Z{\mathbb Z}
\def\Q{\mathbb Q}
\def\mod{\textit{\emph{~mod~}}}
\def\R{\mathcal R}
\def\S{\mathcal S}
\def\*  C{{*  \mathcal C}} 
\def\C{\mathcal C}
\def\D{\mathcal D}
\def\J{\mathcal J}
\def\M{\mathcal M}
\def\T{\mathcal T}          

\newcommand{\Hom}{{\rm Hom}}
\newcommand{\End}{{\rm End}}
\newcommand{\Ext}{{\rm Ext}}
\newcommand{\Mor}{{\rm Mor}\,}
\newcommand{\Aut}{{\rm Aut}\,}
\newcommand{\Hopf}{{\rm Hopf}\,}
\newcommand{\Ann}{{\rm Ann}\,}
\newcommand{\Ker}{{\rm Ker}\,}
\newcommand{\Reg}{{\rm Reg}\,}
\newcommand{\Coker}{{\rm Coker}\,}
\newcommand{\Img}{{\rm Im}\,}
\newcommand{\coim}{{\rm Coim}\,}
\newcommand{\Trace}{{\rm Trace}\,}
\newcommand{\Char}{{\rm Char}\,}
\newcommand{\Mod}{{\rm mod}}
\newcommand{\Spec}{{\rm Spec}\,}
\newcommand{\sgn}{{\rm sgn}\,}
\newcommand{\Id}{{\rm Id}\,}
\newcommand{\Com}{{\rm Com}\,}
\newcommand{\codim}{{\rm codim}}
\newcommand{\Mat}{{\rm Mat}}
\newcommand{\can}{{\rm can}}
\newcommand{\sign}{{\rm sign}}
\newcommand{\kar}{{\rm kar}}
\newcommand{\rad}{{\rm rad}}

\def\lan{\langle}
\def\ran{\rangle}
\def\ot{\otimes}

\def\id{{\small \textit{\emph{1}}}\!\!1}    
\def\To{{\multimap\!\to}}
\def\bigperp{{\LARGE\textrm{$\perp$}}} 
\newcommand{\QED}{\hspace{\stretch{1}}
\makebox[0mm][r]{$\Box$}\\}

\def\RR{{\mathbb R}}
\def\FF{{\mathbb F}}
\def\NN{{\mathbb N}}
\def\CC{{\mathbb C}}
\def\DD{{\mathbb D}}
\def\ZZ{{\mathbb Z}}
\def\QQ{{\mathbb Q}}
\def\HH{{\mathbb H}}
\def\units{{\mathbb G}_m}
\def\GG{{\mathbb G}}
\def\EE{{\mathbb E}}
\def\FF{{\mathbb F}}
\def\rightact{\hbox{$\leftharpoonup$}}
\def\leftact{\hbox{$\rightharpoonup$}}

\newcommand{\Aa}{\mathcal{A}}
\newcommand{\Bb}{\mathcal{B}}
\newcommand{\Cc}{\mathcal{C}}
\newcommand{\Dd}{\mathcal{D}}
\newcommand{\Ee}{\mathcal{E}}
\newcommand{\Ff}{\mathcal{F}}
\newcommand{\Hh}{\mathcal{H}}
\newcommand{\Ii}{\mathcal{I}}
\newcommand{\Mm}{\mathcal{M}}
\newcommand{\Pp}{\mathcal{P}}
\newcommand{\Rr}{\mathcal{R}}
\def\*  C{{}*  \hspace*  {-1pt}{\Cc}}

\def\text#1{{\rm {\rm #1}}}

\def\smashco{\mathrel>\joinrel\mathrel\triangleleft}
\def\cosmash{\mathrel\triangleright\joinrel\mathrel<}

\def\Nat{\dul{\rm Nat}}

\renewcommand{\subjclassname}{\textup{2000} Mathematics Subject
     Classification}

\newtheorem{prop}{Proposition}[section] 
\newtheorem{lemma}[prop]{Lemma}
\newtheorem{cor}[prop]{Corollary}
\newtheorem{theo}[prop]{Theorem}
                       
\theoremstyle{definition}
\newtheorem{Def}[prop]{Definition}
\newtheorem{ex}[prop]{Example}
\newtheorem{exs}[prop]{Examples}
\newtheorem{Not}[prop]{Notation}
\newtheorem{Ax}[prop]{Axiom}
\newtheorem{rems}[prop]{Remarks}
\newtheorem{rem}[prop]{Remark}
\newtheorem{op}[prop]{Open problem}
\newtheorem{conj}[prop]{Conjecture}

\def\smashco{\mathrel>\joinrel\mathrel\triangleleft}
\def\curlarrow{\mathrel\sim\joinrel\mathrel>}

\title{State pseudo equality algebras}

\author[Lavinia Corina Ciungu]{Lavinia Corina Ciungu} 

\begin{abstract}
Pseudo equality algebras were initially introduced by Jenei and $\rm K\acute{o}r\acute{o}di$ as a possible algebraic semantic for fuzzy type theory, and they have been revised by Dvure\v censkij and Zahiri under the name of JK-algebras. 
The aim of this paper is to investigate the internal states and the state-morphisms on pseudo equality algebras.
We define and study new classes of pseudo equality algebras, such as 
commutative, symmetric, pointed and compatible pseudo equality algebras. 
We prove that any internal state (state-morphism) on a pseudo equality algebra is also an internal state (state-morphism) on its corresponding pseudo BCK(pC)-meet-semilattice, and we prove the converse for the case of linearly ordered symmetric pseudo equality algebras. We also show that any internal state (state-morphism) on a 
pseudo BCK(pC)-meet-semilattice is also an internal state (state-morphism) on its corresponding pseudo equality algebra. The notion of a Bosbach state on a pointed pseudo equality algebra is introduced and it is proved that 
any Bosbach state on a pointed pseudo equality algebra is also a Bosbach state on its corresponding pointed 
pseudo BCK(pC)-meet-semilattice. For the case of an invariant pointed pseudo equality algebra, we show that 
the Bosbach states on the two structures coincide. \\  

\textbf{AMS classification (2000):} 03G25, 06F05, 06F35
\end{abstract}

\maketitle

\section{Introduction}

\emph{Fuzzy type theory} (FTT) has been developed by V. $\rm Nov\acute{a}k$ (\cite{Nov1}) as a fyzzy logic of higher order, the fuzzy version of the classical type theory of the classical logic of higher order.
Other formal systems of FTT have also been described by  V. $\rm Nov\acute{a}k$, and all these models are \emph{implication-based}, while the models of the classical type theory are \emph{equality based} having the identity (equality) as the principal connective. 
Since the first algebraic models for the set of truth values of FTT are residuated lattices, their basic operations are $\wedge$ (meet), $\vee$ (join), $\odot$ (multiplication) and $\rightarrow$ (residuum). In fuzzy logic the last operation is a semantic interpretation of the implication, while the logical equivalence is intepreted by the biresiduum $x\leftrightarrow y=(x\rightarrow y)\wedge (y\rightarrow x)$. 
Thus a basic connective has a semantic interpretation by a derived operation.
In order to overcome this discrepancy, we need a specific algebra of truth values for the fuzzy type theory. 
The first version of such an algebra has been introduced by V. $\rm Nov\acute{a}k$ (\cite{Nov2}) under 
the name of \emph{EQ-algebra} and a new concept of fuzzy type theory has been developed based on EQ-algebras (\cite{Nov9}). 
A fuzzy-equality based logic called \emph{EQ-logic} has also been introduced (\cite{Nov10}), 
while the EQ-logics with delta connective were defined and investigated in \cite{Dyba1}. \\
According to \cite{Nov9}, a non-commutative EQ-algebra is an algebra $(E, \wedge, \odot, \thicksim, 1)$ of the type $(2, 2, 2, 0)$ such that the following axioms are fulfilled for all $x, y, z, u\in E$: \\
$(E_1)$ $(E,\wedge,1)$ is a commutative idempotent monoid w.r.t $\le$ ($x\le y$ defined as $x\wedge y=x),$ \\
$(E_2)$ $(E,\odot,1)$ is a monoid such that the operation $\odot$ is isotone w.r.t. $\le,$  \\
$(E_3)$ $x\thicksim x=1,$ 
$\hspace*{10.7cm}$ (\emph{reflexivity}) \\
$(E_4)$ $((x\wedge y)\thicksim z)\odot (u\thicksim x)\le z\thicksim(u\wedge y),$ 
$\hspace*{6cm}$ (\emph{substitution}) \\
$(E_5)$ $(x\thicksim y)\odot (z\thicksim u)\le (x\thicksim z)\thicksim (y\thicksim u),$ 
$\hspace*{6cm}$ (\emph{congruence}) \\
$(E_6)$ $(x\wedge y\wedge z)\thicksim x\le (x\wedge y)\thicksim x,$ 
$\hspace*{5cm}$ (\emph{isotonicity of implication}) \\
$(E_7)$ $(x\wedge y)\thicksim x\le (x\wedge y\wedge z)\thicksim (x\wedge z),$ 
$\hspace*{4cm}$ (\emph{antitonicity of implication}) \\
$(E_8)$ $x\odot y\le x\thicksim y$. 
$\hspace*{9.8cm}$ (\emph{boundedness}) \\
An EQ-algebra is \emph{commutative} if $\odot$ is commutative. \\
The operation $\thicksim$ is a fuzzy equality and the implication $\rightarrow$ is defined by 
$x\rightarrow y=(x\wedge y)\thicksim x$, hence the tie between multiplication and residuation is weaker than in the 
case of residuated lattices. In this sense, EQ-algebras generalize the residuated lattices. \\
As S. Jenei mentioned in \cite{Jen2}, if the product operation in EQ-algebras is replaced by another binary operation 
smaller or equal than the original product we still obtain an EQ-algebra, and this fact might make it difficult to 
obtain certain algebraic results. \\
For this reason, S. Jenei introduced in \cite{Jen2} a new structure, called \emph{equality algebra} consisting of two 
binary operations - meet and equivalence, and constant $1$. 
It was proved in \cite{Jen1}, \cite{Ciu1} that any equality algebra has a corresponding BCK-meet-semilattice 
satisfying the \emph{contraction} condition (BCK(C)-meet-semilattice, for short)  
and any BCK(C)-meet-semilattice has a corresponding equality algebra. 
Since the equality algebras could also be ``candidates" for a possible algebraic semantics for fuzzy type theory, 
their study is highly motivated. 
As a generalization of equality algebras, Jenei and $\rm K\acute{o}r\acute{o}di$ introduced in \cite{Jen1} 
a concept of \emph{pseudo equality algebras} and proved that the pseudo equality algebras are term equivalent to 
pseudo BCK-meet-semilattices. 
In \cite{Ciu1} a gap was found in the proof of this result and a counterexample was given as well as a correct 
version of it. The correct version of the corresponding result for equality algebras was also proved. 
Moreover, Dvure\v censkij and Zahiri showed in \cite{Dvu7} that every pseudo equality algebra in the sense of 
\cite{Jen1} is an equality algebra and they defined and investigated a new concept of pseudo equality algebras 
(called JK-algebras) and established a connection between pseudo equality algebras and a special class of pseudo BCK-meet-semilattices (pseudo BCK(pC)-meet-semilattices).  
Apart from their logical interest, equality algebras as well as pseudo equality algebras seem to have important algebraic properties and it is worth studying them from an algebraic point of view. \\

In 1995 D. Mundici introduced an analogue of the probability measure on MV-algebras, called a \emph{state} (\cite{Mun1}), as averaging process for formulas in \L ukasiewicz logic. After that, the states on other many-valued logic algebras have been intensively studied.  
Flaminio and Montagna were the first to present a unified approach to states and probabilistic many-valued logic in a 
logical and algebraic setting (\cite{Fla1}). They added a unary operation, called \emph{internal state} or 
\emph{state operator} to the language of MV-algebras which preserves the usual properties of states. 
A more powerful type of logic can be given by algebraic structures with internal states, and they are also very interesting varieties of universal algebras. 
Di Nola and Dvure\v censkij introduced the notion of state-morphism MV-algebra which is a stronger variation of state MV-algebra (\cite{DiDv1}, \cite{DiDv1-e}).
The state BCK-algebras and state-morphim BCK-algebras were defined and studied in \cite{Bor1}, and 
recently the state operators and the state-morphism operators on equality algebras have been introduced and 
investigated in \cite{Ciu5}. \\

In this paper we define and study the internal states of type I and type II and the state-morphisms on pseudo 
equality algebras. 
We investigate the connections between internal states and state-morphisms on a pseudo equality algebra and those 
on its corresponding pseudo BCK(pC)-meet-semilattice. 
We prove that any internal state (state-morphism) on a pseudo equality algebra is also an internal state 
(state-morphism) on its corresponding pseudo BCK(pC)-meet-semilattice, and we prove the converse for the case of linearly ordered symmetric pseudo equality algebras. We also show that any internal state (state-morphism) on a 
pseudo BCK(pC)-meet-semilattice is also an internal state (state-morphism) on its corresponding pseudo equality algebra. 
The notion of a Bosbach state on a pointed pseudo equality algebra is introduced and it is proved that 
any Bosbach state on a pointed pseudo equality algebra is also a Bosbach state on its corresponding pointed 
pseudo BCK(pC)-meet-semilattice. For the case of an invariant pointed pseudo equality algebra, we show that 
the Bosbach states on the two structures coincide.
The classes of commutative, symmetric, pointed and compatible pseudo equality algebras 
are defined and their properties are investigated. 
We prove that in the case of commutatice pseudo equality algebras the two types of internal states coincide 
and we also show that a pseudo-hoop is a compatible pseudo equality algebra. 
We prove that a state-morphism on the set of regular elements on a compatible pseudo equality algebra $A$ can be extended to a state morphism on $A$. Additionally, we give new properties of pseudo equality algebras.

$\vspace*{5mm}$

\section{Pseudo equality algebras}

Pseudo equality algebras have been firstly defined by Jenei and $\rm K\acute{o}r\acute{o}di$ in \cite{Jen1} 
as a generalization of equality algebras. 
Dvure\v censkij and Zahiri showed in \cite{Dvu7} that every pseudo equality algebra in the sense of 
\cite{Jen1} is an equality algebra and they defined and investigated a new concept of pseudo equality algebras 
(JK-algebras) and established a connection between pseudo equality algebras and a special class of pseudo BCK-meet-semilattices.
In this section we recall the main notions and results and we present new properties of pseudo equality algebras. 

\begin{Def} \label{ps-eq-05} $\rm($\cite{Dvu7}$\rm)$
A \emph{pseudo equality algebra} (or a \emph{JK-algebra}) is an algebra 
$\mathcal{A}=(A, \wedge, \thicksim, \backsim, 1)$ of the type $(2, 2, 2, 0)$ such that the following axioms are fulfilled for all $x, y, z\in A$: \\
$(A_1)$ $(A, \wedge, 1)$ is a meet-semilattice with top element $1,$ \\
$(A_2)$ $x \thicksim x = x \backsim x = 1,$ \\ 
$(A_3)$ $x \thicksim 1 = 1 \backsim x = x,$ \\
$(A_4)$ $x \le y \le z$ implies  
$x \thicksim z \le y \thicksim z$, $x \thicksim z \le x \thicksim y$, $z \backsim x \le z \backsim y$ and  
$z \backsim x \le y \backsim x,$ \\ 
$(A_5)$ $x \thicksim y \le (x \wedge z) \thicksim (y \wedge z)$  and 
        $x \backsim y \le (x \wedge z) \backsim (y \wedge z),$ \\  
$(A_6)$ $x \thicksim y \le (z \thicksim x) \backsim (z \thicksim y)$ and 
        $x \backsim y \le (x \backsim z) \thicksim(y \backsim z),$ \\
$(A_7)$ $x \thicksim y \le (x \thicksim z) \thicksim (y \thicksim z)$ and 
        $x \backsim y \le (z \backsim x) \backsim (z \backsim y)$. 
\end{Def}        
        
The operation $\wedge$ is called \emph{meet}(\emph{infimum}) and $\thicksim$, $\backsim$ are called 
\emph{equality} operations. We write $x \le y$ (and $y \ge x$) iff $x \wedge y = x$. \\
In the algebra $\mathcal{A}$ other two operations are defined, called \emph{implications}: \\
$\hspace*{5cm}$ $x \rightarrow y = (x\wedge y) \thicksim x$ \\
$\hspace*{5cm}$ $x \rightsquigarrow y = x \backsim (x \wedge y)$. \\
In the sequel we will also refer to the pseudo equality algebra $(A, \wedge, \thicksim, \backsim, 1)$ 
by its universe $A$. 
We will agree that $\thicksim$, $\backsim$, $\rightarrow$ and $\rightsquigarrow$ have higher priority than the operation $\wedge$. \\ 
A pseudo equality algebra $A$ is called \emph{bounded} if there exists an element $0\in A$ such that $0\le x$ 
for all $x\in A$. A bounded pseudo equality algebra is denoted by $(A, \wedge, \thicksim, \backsim, 0, 1)$.  

\begin{prop} \label{ps-eq-05-10} $\rm($\cite{Dvu7}$\rm)$ 
In any pseudo equality algebra $(A, \wedge, \thicksim, \backsim, 1)$ the following hold for all $x, y, z\in A$: \\
$(1)$ $x\wedge y\thicksim x\le x\wedge y\wedge z\thicksim x\wedge z$ and 
      $x\rightarrow y\le x\wedge z\le x\wedge z\rightarrow y;$ \\ 
$(2)$ $x\backsim x\wedge y\le x\wedge z\backsim x\wedge y\wedge z$ and 
      $x\rightsquigarrow y\le x\wedge z\rightsquigarrow y$. 
\end{prop}

\begin{prop} \label{ps-eq-10} $\rm($\cite{Dvu7}$\rm)$ 
In any pseudo equality algebra $(A, \wedge, \thicksim, \backsim, 1)$ the following hold for all $x, y, z\in A$: \\
$(1)$ $x\thicksim y\le y\rightarrow x$ and $x\backsim y\le x\rightsquigarrow y;$ \\
$(2)$ $x\le ((y\thicksim x)\backsim y)\wedge (y \thicksim (x\backsim y));$ \\
$(3)$ $x\backsim y=1$ or $y\thicksim x=1$ imply $x\le y;$ \\
$(4)$ $x\thicksim y=1$ implies $z\thicksim x\le z\thicksim y$ and 
      $x\backsim y=1$ implies $y\backsim z\le x\backsim z;$ \\
$(5)$ $x\le y$ iff $x\rightarrow y=1$ iff $x\rightsquigarrow y=1;$ \\
$(6)$ $1\rightarrow x=1\rightsquigarrow x=x$, 
      $x\rightarrow 1=x\rightsquigarrow 1=x\rightarrow x=x\rightsquigarrow x=1$, 
      $x\rightarrow x=x\rightsquigarrow x=1;$ \\
$(7)$ $x\le (y \rightarrow x)\wedge (y \rightsquigarrow x);$ \\
$(8)$ $x\le ((x \rightarrow y)\rightsquigarrow y)\wedge ((x \rightsquigarrow y)\rightarrow y);$ \\
$(9)$ $x\rightarrow y\le (y\rightarrow z)\rightsquigarrow (x\rightarrow z)$ and 
      $x\rightsquigarrow y\le (y\rightsquigarrow z)\rightarrow (x\rightsquigarrow z);$ \\  
$(10)$ $x\le y\rightarrow z$ iff $y\le x\rightsquigarrow z;$ \\
$(11)$ $x\rightarrow (y\rightsquigarrow z)=y\rightsquigarrow (x\rightarrow z);$ \\
$(12)$ $x\rightarrow y\le (x\wedge z)\rightarrow (y\wedge z)$ and  
       $x\rightsquigarrow y\le (x\wedge z)\rightsquigarrow (y\wedge z);$ \\
$(13)$ $x\rightarrow y=x\rightarrow (x\wedge y)$ and $x\rightsquigarrow y=x\rightsquigarrow (x\wedge y);$ \\
$(14)$ $1\thicksim x=x\backsim 1;$ \\
$(15)$ if $x\le y$, then $x\le (x\thicksim y)\wedge (y\backsim x);$ \\
$(16)$ $x\thicksim y\le 1\thicksim (y\thicksim x)$ and $x\backsim y\le 1\thicksim (y\backsim x)$. 
\end{prop}

\begin{prop} \label{ps-eq-20}
In any pseudo equality algebra $(A, \wedge, \thicksim, \backsim, 1)$ the following hold for all $x, y\in A$:\\
$(1)$ $y\le (x\wedge y\thicksim x)\wedge (x\backsim x\wedge y);$ \\ 
$(2)$ if $x\le y$, then $x\le (x\thicksim y)\wedge (y\backsim x);$ \\
$(3)$ $x\le ((x\wedge y\thicksim x)\backsim y)\wedge (y\thicksim (x\backsim x\wedge y));$ \\
$(4)$ $y\le ((x\wedge y\thicksim x)\backsim y)\wedge (y\thicksim (x\backsim x\wedge y));$ \\
$(5)$ $x\thicksim y\le x\wedge y\thicksim y$ and $x\backsim y\le x\backsim x\wedge y$. 
\end{prop}
\begin{proof}
$(1)$ It follows by Proposition \ref{ps-eq-10}$(7)$. \\
$(2)$ Applying $(1)$ it follows that 
$x\le (x\wedge y\thicksim y)\wedge (y\backsim x\wedge y)=(x\thicksim y)\wedge (y\backsim x);$ \\
$(3)$ Applying $(A_6)$ we have: \\
$\hspace*{2cm}$
$x=x\thicksim 1 \le (x\wedge y\thicksim x) \backsim (x\wedge y\thicksim 1)=
(x\wedge y\thicksim x)\backsim x\wedge y$. \\
From $x\wedge y\le y\le x\wedge y\thicksim x$, applying $(A_4)$ we get 
$(x\wedge y \thicksim x)\backsim x\wedge y\le (x\wedge y \thicksim x)\backsim y$, so 
$x\le (x\wedge y\thicksim x)\backsim y$. 
Similarly $x=1\backsim x\le (1\backsim x\wedge y)\thicksim (x\backsim x\wedge y)=
x\wedge y\thicksim (x\backsim x\wedge y)$. \\
Since $x\wedge y\le y\le x\backsim x\wedge y$, applying $(A_4)$ we get 
$x\wedge y\thicksim (x\backsim x\wedge y)\le y\thicksim (x\backsim x\wedge y)$, hence 
$x\le y\thicksim (x\backsim x\wedge y)$. \\
$(4)$ From $y\le x\wedge y \thicksim x$ and $y\le x\backsim x\wedge y$ , applying $(2)$ we get 
$y\le (x\wedge y\thicksim x)\backsim y$ and $y\le y\thicksim (x\backsim x\wedge y)$, respectively. \\
$(5)$ It is a consequence of Proposition \ref{ps-eq-10}$(1)$.
\end{proof}

\begin{prop} \label{ps-eq-20-10}
Let $(A, \wedge, \thicksim, \backsim, 1)$ be a pseudo equality algebra and let $x, y\in A$ such that $x\le y$.  Then the following hold for all $z\in A$:\\
$(1)$ $y\wedge z\thicksim y\le x\wedge z\thicksim x$ and $y\backsim y\wedge z\le x\backsim x\wedge z;$ \\ 
$(2)$ $z\wedge x\thicksim z\le z\wedge y\thicksim z$ and $z\backsim z\wedge x\le z\backsim z\wedge y$.
\end{prop}
\begin{proof}
$(1)$ By Proposition \ref{ps-eq-05-10}$(1)$ we have
$y\wedge z\thicksim y\le x\wedge y\wedge z\thicksim x\wedge y=x\wedge z\thicksim x$. \\
Similarly from Proposition \ref{ps-eq-05-10}$(2)$ it follows that 
$y\backsim y\wedge z\le x\wedge y\backsim x\wedge y\wedge z=x\backsim x\wedge z$. \\
$(2)$ From $x\wedge y\wedge z\le y\wedge z\le z$, applying $(A_4)$ we have 
      $x\wedge y\wedge z\thicksim z\le y\wedge z\thicksim z$ and 
      $z\backsim x\wedge y \wedge z\le z\backsim y\wedge z$, that is 
      $z\wedge x\thicksim z\le z\wedge y\thicksim z$ and $z\backsim z\wedge x\le z\backsim z\wedge y$.
\end{proof}

\begin{prop} \label{ps-eq-30}
Let $(A, \wedge, \thicksim, \backsim, 1)$ be a pseudo equality algebra. Then the following hold for all $x, y\in A$: \\
$(1)$ $y\thicksim ((x\wedge y\thicksim x)\backsim y)=x\wedge y\thicksim x;$ \\
$(2)$ $(y\thicksim (x\backsim x\wedge y))\backsim y=x\backsim x\wedge y$. 
\end{prop}
\begin{proof}
$(1)$ By Proposition \ref{ps-eq-20}$(3)$, $x\le (x\wedge y\thicksim x)\backsim y$. 
Taking $z:=y$ and $y:=(x\wedge y\thicksim x)\backsim y$ in Proposition \ref{ps-eq-20-10}$(1)$ we get 
$((x\wedge y\thicksim x)\backsim y)\wedge y\thicksim ((x\wedge y\thicksim x)\backsim y)\le x\wedge y\thicksim x$. 
Since by Proposition \ref{ps-eq-20}$(4)$, $y\le (x\wedge y\thicksim x)\backsim y$, it follows that 
$y\thicksim ((x\wedge y\thicksim x)\backsim y)\le x\wedge y\thicksim x$. \\
On the other hand, by Proposition \ref{ps-eq-10}$(2)$ 
$x\wedge y\thicksim x\le y\thicksim ((x\wedge y\thicksim x)\backsim y)$, hence  
$y\thicksim ((x\wedge y\thicksim x)\backsim y)=x\wedge y\thicksim x$. \\
$(2)$ Similarly from $x\le y\thicksim (x\backsim x\wedge y)$, applying Propositions \ref{ps-eq-20-10}$(1)$ and 
\ref{ps-eq-20}$(4)$ we get $(y\thicksim (x\backsim x\wedge y))\backsim y\le x\backsim x\wedge y$.
By Proposition \ref{ps-eq-10}$(2)$, $x\backsim x\wedge y\le (y\thicksim (x\backsim x\wedge y))\backsim y$, so  
$(y\thicksim (x\backsim x\wedge y))\backsim y=x\backsim x\wedge y$. 
\end{proof}


Pseudo BCK-algebras were introduced by G. Georgescu and A. Iorgulescu in \cite{Geo15} as algebras 
with "two differences", a left- and right-difference, instead of one $*$ and with a constant element $0$ as the least element. Nowadays pseudo BCK-algebras are used in a dual form, with two implications, $\to$ and $\rightsquigarrow$ and with one constant element $1$, that is the greatest element. Thus such pseudo BCK-algebras are in the "negative cone" and are also called "left-ones". \\
A \emph{pseudo BCK-algebra} (more precisely, \emph{reversed left-pseudo BCK-algebra}) is a structure 
${\mathcal B}=(B,\leq,\rightarrow,\rightsquigarrow,1)$ where $\leq$ is a binary relation on $B$, $\rightarrow$ and $\rightsquigarrow$ are binary operations on $B$ and $1$ is an element of $B$ satisfying, for 
all $x, y, z \in B$, the  axioms:\\
$(B_1)$ $x \rightarrow y \leq (y \rightarrow z) \rightsquigarrow (x \rightarrow z)$, $\:\:\:$ 
            $x \rightsquigarrow y \leq (y \rightsquigarrow z) \rightarrow (x \rightsquigarrow z);$ \\ 
$(B_2)$ $x \leq (x \rightarrow y) \rightsquigarrow y$,$\:\:\:$ $x \leq (x \rightsquigarrow y) \rightarrow y;$ \\
$(B_3)$ $x \leq x;$ \\
$(B_4)$ $x \leq 1;$ \\
$(B_5)$ if $x \leq y$ and $y \leq x$, then $x = y;$ \\
$(B_6)$ $x \leq y$ iff $x \rightarrow y = 1$ iff $x \rightsquigarrow y = 1$. \\

Since the partial order $\leq$ is determined by either of the two ``arrows", we can eliminate $\leq$ from the signature 
and denote a pseudo BCK-algebra by ${\mathcal B}=(X,\rightarrow,\rightsquigarrow,1)$. \\
An equivalent definition of a pseudo BCK-algebra is given in \cite{Kuhr5}. \\
The structure ${\mathcal B}=(B,\rightarrow,\rightsquigarrow,1)$ of the type $(2,2,0)$ is a pseudo BCK-algebra iff it satisfies the following identities and quasi-identity, for all $x, y, z \in B$:\\ 
$(B_1')$ $(x \rightarrow y) \rightsquigarrow [(y \rightarrow z) \rightsquigarrow (x \rightarrow z)]=1;$ \\
$(B_2')$ $(x \rightsquigarrow y) \rightarrow [(y \rightsquigarrow z) \rightarrow (x \rightsquigarrow z)]=1;$ \\
$(B_3')$ $1 \rightarrow x = x;$ \\
$(B_4')$ $1 \rightsquigarrow x = x;$ \\
$(B_5')$ $x \rightarrow 1 = 1;$ \\
$(B_6')$ $(x\rightarrow y =1$ and $y \rightarrow x=1)$ implies $x=y$. \\
The partial order $\le$ is defined by $x \le y$ iff $x \rightarrow y =1$ (iff $x \rightsquigarrow y =1$). \\ 
If the poset $(B, \le)$ is a meet-semilattice, then ${\mathcal B}$ is called a \emph{pseudo BCK-meet-semilattice} and 
we denote it by $\mathcal{B}=(B, \wedge, \rightarrow, \rightsquigarrow, 1)$.       
If $(B,\leq)$ is a lattice, then we will say that ${\mathcal B}$ is a \emph{pseudo BCK-lattice} and 
it is denoted by $\mathcal{B}=(B, \wedge, \vee, \rightarrow, \rightsquigarrow, 1)$. \\

A pseudo BCK-algebra $\mathcal{B}=(B,\rightarrow,\rightsquigarrow,1)$ with a constant $a\in B$ (which 
can denote any element) is called a \emph{pointed pseudo BCK-algebra}.\\
A pointed pseudo BCK-algebra is denoted by $\mathcal{B}=(B,\rightarrow,\rightsquigarrow,a,1)$. 

A pseudo BCK-algebra $B$ is called \emph{bounded} if there exists an element $0\in B$ such that $0\le x$ 
for all $x\in B$. In a bounded pseudo BCK-algebra $(B, \rightarrow, \rightsquigarrow, 0, 1)$ we can define 
two negations: $x^{\rightarrow_0}=x\rightarrow 0$ and $x^{\rightsquigarrow_0}=x\rightsquigarrow 0$. 
A bounded pseudo BCK-algebra $B$ is called \emph{good} if it satisfies the identity 
$x^{\rightarrow_0\rightsquigarrow_0}=x^{\rightsquigarrow_0\rightarrow_0}$ for all $x\in B$. 

\begin{lemma} \label{ps-eq-40} $\rm($\cite{Geo15}$\rm)$ In any pseudo BCK-algebra 
$(B,\rightarrow,\rightsquigarrow,1)$ the following hold for all $x, y, z\in B$: \\
$(1)$ $x\le y$ implies $z\rightarrow x \le z\rightarrow y$ and $z\rightsquigarrow x \le z\rightsquigarrow y;$ \\
$(2)$ $x\le y$ implies $y\rightarrow z \le x\rightarrow z$ and $y\rightsquigarrow z \le x\rightsquigarrow z;$ \\
$(3)$ $x\rightarrow y \le (z\rightarrow x)\rightarrow (z\rightarrow y)$ and 
      $x\rightsquigarrow y \le (z\rightsquigarrow x)\rightsquigarrow (z\rightsquigarrow y);$ \\
$(4)$ $x\rightarrow (y\rightsquigarrow z)=y\rightsquigarrow (x\rightarrow z)$ and 
      $x\rightsquigarrow (y\rightarrow z)=y\rightarrow (x\rightsquigarrow z);$ \\
$(5)$ $x\le y\rightarrow x$ and $x\le y\rightsquigarrow x$.       
\end{lemma}

For more details about the properties of a pseudo BCK-algebra we refer te reader to \cite{Ior14} and \cite{Ciu2}. \\
Let $B$ be a pseudo BCK-algebra. The subset $D \subseteq B$ is called a \emph{deductive system} of $B$ if 
it satisfies the following conditions:\\
$(i)$ $1 \in D;$ \\
$(ii)$ for all $x, y \in B$, if $x, x \rightarrow y \in D$, then $y \in D$. \\ 
Condition $(ii)$ is equivalent to the condition: \\
$(ii^{\prime})$ for all $x, y \in B$, if $x, x \rightsquigarrow y \in D$, then $y \in D$. \\ 

A deductive system $D$ of a pseudo BCK-algebra $B$ is said to be \emph{normal} if it satisfies the condition:\\
$(iii)$ for all $x, y \in B$, $x \rightarrow y \in D$ iff $x \rightsquigarrow y \in D$. \\
We will denote by ${\mathcal DS}_{BCK}(B)$ the set of all deductive systems and by ${\mathcal DS}_{n_{BCK}}(B)$ 
the set of all normal deductive systems of a pseudo BCK-algebra $B$. \\
Obviously $\{1\}, B \in {\mathcal DS}_{BCK}(B), {\mathcal DS}_{n_{BCK}}(B)$ and 
${\mathcal DS}_{n_{BCK}}(B)\subseteq {\mathcal DS}_{BCK}(B)$.
For every subset $X\subseteq B$, the smallest deductive system of $B$ containing $X$ (i.e. the intersection of all deductive systems $D\in{\mathcal DS}_{BCK}(B)$ such that $X\subseteq D$) is called the deductive 
system \emph{generated by} $X$ and it will be denoted by $[X)$. If $X=\{x\}$ we write $[x)$ instead of $[\{x\})$.

\begin{Def} \label{ps-eq-60}
A pseudo BCK-meet-semilattice with the \emph{$\rm($pD$\rm)$ condition} (i.e. with the  \emph{pseudo-distributivity} condition) or a \emph{pseudo BCK$\rm($pD$\rm)$-meet-semilattice} for short, is a pseudo BCK-meet-semilattice 
$(X, \wedge, \rightarrow, \rightsquigarrow, 1)$ satisfying the (pD) condition:\\
(pD) $\hspace*{2.1cm}$   $x\rightarrow (y\wedge z)=(x\rightarrow y)\wedge (x\rightarrow z)$, \\
     $\hspace*{3cm}$ $x\rightsquigarrow (y\wedge z)=(x\rightsquigarrow y)\wedge (x\rightsquigarrow z)$ \\ 
for all $x, y, z\in X$. 
\end{Def}

A pseudo BCK-algebra with the \emph{$\rm($pP$\rm)$ condition} (i.e. with the  \emph{pseudo-product} condition) or 
a \emph{pseudo BCK$\rm($pP$\rm)$-algebra} for short, is a pseudo BCK-algebra 
$(X, \leq,\rightarrow, \rightsquigarrow, 1)$ satisfying the (pP) condition:\\
(pP) For all $x, y \in X$, there exists  \\
$\hspace*{3cm}$ $x\odot y=\min\{z \mid x \leq y \rightarrow z\}=\min\{z \mid y \leq x \rightsquigarrow z\}$. \\
A pseudo BCK(pP)-algebra is denoted by $(X, \odot, \rightarrow, \rightsquigarrow, 1)$. \\
It was proved in \cite{Ior1} that the (pP) condition is equivalent to the \emph{pseudo-residuation property} 
((pRP) for short):\\
(pRP) For all $x, y, z$ the following hold:\\
$\hspace*{3cm}$ $x\odot y\le z$ iff $x\le y\rightarrow z$ iff $y\le x\rightsquigarrow z$.  

Every pseudo BCK(pP)-meet-semilattice $(X, \wedge, \odot, \rightarrow, \rightsquigarrow, 1)$ is a pseudo BCK(pD)-meet-semilattice (see \cite{Ciu1}).  

\begin{lemma} \label{ps-eq-70-10} Every deductive system of a pseudo BCK(pD)-meet-semilattice $B$ is a 
subalgebra of $B$.
\end{lemma}
\begin{proof} Let $(B,\wedge,\rightarrow,\rightsquigarrow,1)$ be a pseudo BCK(pD)-meet-semilattice,  
$D\in{\mathcal DS}_{BCK}(B)$ and $x, y\in D$. 
From $y\le x\rightarrow y,x\rightsquigarrow y$, we get $x\rightarrow y,x\rightsquigarrow y\in D$. 
Applying the (pD) property, we have $x\rightarrow x\wedge y=x\rightarrow y\in D$. Since $x\in D$, it follows that 
$x\wedge y\in D$. Moreover, $1\in D$, thus $D$ is a subalgebra of $B$.
\end{proof}

Pseudo-hoops were introduced in \cite{Geo16} as a generalization of hoops which were originally defined and studied   by Bosbach in \cite{Bos1} and \cite{Bos2} under the name of ``residuated integral monoids". 
A \emph{pseudo-hoop} is an algebra $(A, \odot, \rightarrow, \rightsquigarrow, 1)$ of the type $(2,2,2,0)$ such that, for all $x,y,z \in A$:\\
$(PH_1)$ $x\odot 1=1\odot x=x;$\\
$(PH_2)$ $x\rightarrow x=x\rightsquigarrow x=1;$\\
$(PH_3)$ $(x\odot y) \rightarrow z=x\rightarrow (y\rightarrow z);$\\
$(PH_4)$ $(x\odot y) \rightsquigarrow z=y\rightsquigarrow (x\rightsquigarrow z);$\\
$(PH_5)$ $(x\rightarrow y)\odot x=(y\rightarrow x)\odot y=x\odot(x\rightsquigarrow y)=y\odot(y\rightsquigarrow x)$. 

A pseudo-hoop $A$ is bounded if there exists an element $0\in A$ such that $0\le x$ for all $x\in A$. \\
It was proved that a pseudo-hoop has the pseudo-divisibility condition and it is a meet-semilattice with 
$x\wedge y=(x\rightarrow y)\odot x=(y\rightarrow x)\odot y=x\odot(x\rightsquigarrow y)=y\odot(y\rightsquigarrow x)$.
It follows that a bounded R$\ell$-monoid can be viewed as a bounded pseudo-hoop together with the join-semilattice property. In other words, a pseudo-hoop is a meet-semilattice ordered residuated, integral and divisible monoid. 

\begin{rems} \label{ps-eq-80} $\rm($\cite{{Ciu2}}$\rm)$
$(1)$ Every linearly ordered pseudo BCK-algebra is a pseudo BCK(pP)-algebra. 
Thus every linearly ordered pseudo BCK-meet-semilattice is a pseudo BCK(pD)-meet-semilattice. \\
$(2)$ Since every pseudo-hoop $A$ is a pseudo BCK(pP)-algebra, a partial order $\le$ on $A$ can be 
defined in the same way as in the case of pseudo BCK-algebras.
Moreover, $(A,\le)$ is a meet-semilattice with 
$x\wedge y=(x\rightarrow y)\odot x=(y\rightarrow x)\odot y=x\odot(x\rightsquigarrow y)=y\odot(y\rightsquigarrow x)$. \\
It follows that every pseudo-hoop is a pseudo BCK(pD)-meet-semilattice.
\end{rems}

\begin{Def} \label{ps-eq-90}
A pseudo BCK-meet-semilattice with the \emph{$\rm($pC$\rm)$ condition} (i.e. with the  \emph{pseudo-contraction} condition) or a \emph{pseudo BCK$\rm($pC$\rm)$-meet-semilattice} for short, is a pseudo BCK-meet-semilattice 
$(X, \wedge, \rightarrow, \rightsquigarrow, 1)$ satisfying the (pC) condition:\\
(pC) $\hspace*{2.1cm}$   $x\rightarrow y\le (x\wedge z)\rightarrow (y\wedge z)$, \\
     $\hspace*{3cm}$ $x\rightsquigarrow y\le (x\wedge z)\rightsquigarrow (y\wedge z)$ \\ 
for all $x, y, z\in X$. 
\end{Def}

\begin{prop} \label{ps-eq-100} Any pseudo BCK(pD)-meet-semilattice is a pseudo BCK(pC)-meet-semilattice.
\end{prop}
\begin{proof}
Let $(X, \wedge, \rightarrow, \rightsquigarrow, 1)$ be a pseudo BCK(pD)-meet-semilattice. 
From $z\wedge x\le x $ we get: \\
$\hspace*{1.0cm}$
$x\rightarrow y\le (z\wedge x)\rightarrow y$ \\
$\hspace*{2.2cm}$
$=((z\wedge x)\rightarrow (z\wedge x))\wedge ((z\wedge x) \rightarrow y)$ \\
$\hspace*{2.2cm}$
$=(z\wedge x)\rightarrow (z\wedge x\wedge y)$ \\
$\hspace*{2.2cm}$
$=(z\wedge x)\rightarrow ((z\wedge x)\wedge (z\wedge y))$ \\
$\hspace*{2.2cm}$
$=((z\wedge x)\rightarrow (z\wedge x)) \wedge ((z\wedge x)\rightarrow (z\wedge y))$ \\
$\hspace*{2.2cm}$ 
$=(z\wedge x)\rightarrow (z\wedge y)$. \\
Similarly $x\rightsquigarrow y\le (x\wedge z)\rightsquigarrow (y\wedge z)$. 
Thus $X$ satisfies the (pC) condition. 
\end{proof}

\begin{rem} \label{ps-eq-100-10} $\rm($\cite{Dvu7}$\rm)$ If $(A, \wedge, \rightarrow, \rightsquigarrow, 1)$ is a 
pseudo BCK(pC)-meet-semilattice, then $x\rightarrow x\wedge y=x\rightarrow y$ and 
$x\rightsquigarrow x\wedge y=x\rightsquigarrow y$.
\end{rem}

The following theorem provides a connection of pseudo equality algebras with the class of 
pseudo BCK(pC)-meet-semilattices. 
 
\begin{theo} \label{ps-eq-110} $\rm($\cite{Dvu7}$\rm)$ The following statements hold: \\
$(1)$  Let $\mathcal{A}=(A, \wedge, \thicksim, \backsim, 1)$ be a pseudo equality algebra.  
Then $\Psi(\mathcal{A})=(A, \wedge, \rightarrow, \rightsquigarrow, 1)$ is a pseudo BCK(pC)-meet-semilattice, where $x\rightarrow y=x\wedge y\thicksim x$ and $x\rightsquigarrow y=x\backsim x\wedge y$ for all $x, y\in A;$ \\ 
$(2)$ Let $\mathcal{B}=(B, \wedge, \rightarrow, \rightsquigarrow, 1)$ be a pseudo BCK(pC)-meet-semilattice. 
Then $\Phi(\mathcal{B})=(B, \wedge, \thicksim, \backsim, 1)$ is a pseudo equality algebra, where 
$x\thicksim y=y\rightarrow x$ and $x\backsim y=x\rightsquigarrow y$ for all $x, y\in B$.             
\end{theo}

\begin{prop} \label{ps-eq-120} If $(A, \wedge, \rightarrow, \rightsquigarrow, 1)$ is a non-trivial 
pseudo BCK(pC)-meet-semilattice, then $\Phi(A)$ is not an equality algebra.
\end{prop}
\begin{proof} 
Let $(A, \wedge, \rightarrow, \rightsquigarrow, 1)$ be a non-trivial pseudo BCK(pC)-meet-semilattice and 
$\Phi(A)=(A, \wedge, \thicksim, \backsim, 1)$ be its corresponding 
pseudo equality algebra. Suppose that $\Phi(A)$ is an equality algebra, that is 
$x\thicksim y=x\backsim y$ for all $x, y\in A$. 
It follows that $x\thicksim y=y\rightarrow x=x\rightsquigarrow y$ for all $x, y\in A$. 
Taking $y=1$ we get $x=1$, so $A$ is a trivial pseudo BCK(pC)-meet-semilattice. 
Hence $\Phi(A)$ is not an equality algebra.  
\end{proof}

For the case of equality algebras, conditions (pD), (pC) become (D) and (C), respectively. \\
The next result is the commutative version of Theorem \ref{ps-eq-110}. 

\begin{theo} \label{ps-eq-130} The following statements hold: \\
$(1)$  Let $\mathcal{A}=(A, \wedge, \thicksim, 1)$ be an equality algebra.  
Then $\Psi(\mathcal{A})=(A, \wedge, \rightarrow, 1)$ is a BCK(C)-meet-semilattice, where 
$x\rightarrow y=x\wedge y\thicksim x$ for all $x, y\in A;$ \\ 
$(2)$ Let $\mathcal{B}=(B, \wedge, \rightarrow, 1)$ be a BCK(C)-meet-semilattice. 
Then $\Phi(\mathcal{B})=(B, \wedge, \thicksim, 1)$ is an equality algebra, where 
$x\thicksim y=y\rightarrow x$ for all $x, y\in B$.             
\end{theo}

With the notations of Theorem \ref{ps-eq-110} we say that a pseudo equality algebra 
$(A, \wedge, \thicksim, \backsim, 1)$ is \emph{invariant} if there exists a pseudo BCK-meet-semilattice 
$(A, \wedge, \rightarrow^{\prime}, \rightsquigarrow^{\prime}, 1)$ such that 
$\Phi((A, \wedge, \rightarrow^{\prime}, \rightsquigarrow^{\prime}, 1))=(A, \wedge, \thicksim, \backsim, 1)$. 

\begin{theo} \label{ps-eq-140} $\rm($\cite{Dvu7}$\rm)$ The following statements hold: \\
$(1)$ Let $(A, \wedge, \thicksim, \backsim, 1)$ be a pseudo equality algebra. 
Then $\Psi(\Phi(\Psi((A, \wedge, \rightarrow, \rightsquigarrow, 1))))=
\Psi((A, \wedge, \rightarrow, \rightsquigarrow, 1));$ \\
$(2)$ Let $(B, \wedge, \rightarrow, \rightsquigarrow, 1)$ be a pseudo BCK(pC)-meet-semilattice. Then 
$\Psi(\Phi((B, \wedge, \rightarrow, \rightsquigarrow, 1)))=(B, \wedge, \rightarrow, \rightsquigarrow, 1);$ \\
$(3)$ A pseudo equality algebra $(A, \wedge, \thicksim, \backsim, 1)$ is invariant if and only if 
$\Phi(\Psi((A, \wedge, \thicksim, \backsim, 1)))=(A, \wedge, \thicksim, \backsim, 1);$ \\
$(4)$ The class of pseudo BCK(pC)-meet-semilattices and the class of invariant pseudo equality  
algebras are term equivalent; \\
$(5)$ The category of pseudo BCK(pC)-meet-semilattices and the category of invariant pseudo equality  
algebras are categorically equivalent.
\end{theo}

\begin{prop} \label{ps-eq-150} A pseudo equality algebra $(A, \wedge, \thicksim, \backsim, 1)$ is invariant if and 
only if $x\wedge y\thicksim y=x\thicksim y$ and $x\backsim x\wedge y=x\backsim y$, for all $x, y\in A$.
\end{prop}
\begin{proof}
According to Theorem \ref{ps-eq-140}, $A$ is invariant if and only if 
$\Phi(\Psi((A, \wedge, \thicksim, \backsim, 1)))=(A, \wedge, \thicksim, \backsim, 1)$. 
Taking into consideration Theorem \ref{ps-eq-110}, 
$\Psi((A, \wedge, \thicksim, \backsim, 1))=(A, \wedge, \rightarrow, \rightsquigarrow, 1)$, where 
$x\rightarrow y=x\wedge y\thicksim x$, $x\rightsquigarrow y=x\backsim x\wedge y$ and 
$\Phi(\Psi((A, \wedge, \thicksim, \backsim, 1)))=(A, \wedge, \thicksim^{\prime}, \backsim^{\prime}, 1)$, where 
$x\thicksim^{\prime} y =y\rightarrow x=x\wedge y\thicksim y$, 
$x\backsim^{\prime} y=x\rightsquigarrow y=x\backsim x\wedge y$ for all $x, y\in A$. \\
Hence we have $(A, \wedge, \thicksim, \backsim, 1)=(A, \wedge, \thicksim^{\prime}, \backsim^{\prime}, 1)$ if 
and only if $x\wedge y\thicksim y=x\thicksim y$ and $x\backsim x\wedge y=x\backsim y$, for all $x, y\in A$.
\end{proof}

\begin{rem} \label{ps-eq-150-10} 
Let $(A, \wedge, \thicksim, \backsim, 1)$ be an invariant pseudo equality algebra and let $x, y\in A$. 
According to Proposition \ref{ps-eq-150}, $x\wedge y\thicksim x=y\thicksim x$ and $x\backsim x\wedge y=x\backsim y$. 
If $x\le y$, then $y\thicksim x=1$ and $x\backsim y=1$. 
On the other hand, by Proposition \ref{ps-eq-10}$(3)$, if $y\thicksim x=1$ or $x\backsim y=1$, then $x\le y$. 
Hence $x\le y$ iff $y\thicksim x=1$ iff $x\backsim y=1$.  
\end{rem}


The commutative pseudo equality algebras have been defined and studied in \cite{Ciu9}. \\
A pseudo equality algebra $(A, \wedge, \thicksim, \backsim, 1)$ is said to be \emph{commutative} if 
the following hold: \\
$\hspace*{3cm}$ $(x\wedge y\thicksim x)\backsim y=(x\wedge y\thicksim y)\backsim x,$ \\
$\hspace*{3cm}$ $y\thicksim (x\backsim x\wedge y)=x\thicksim (y\backsim x\wedge y)$ \\ 
for all $x, y\in A$. \\
Obviously an invariant pseudo equality algebra $(A, \wedge, \thicksim, \backsim, 1)$ is commutative if and only if $(y\thicksim x)\backsim y=(x\thicksim y)\backsim x$ and 
$y\thicksim (x\backsim y)=x\thicksim (y\backsim x)$, for all $x, y\in A$. \\
A pseudo equality algebra $(A, \wedge, \thicksim, \backsim, 1)$ is said to be 
\emph{symmetric pseudo equality algebra} if $x\backsim y=y\thicksim x$ for all $x, y\in A$. 
Obviously any equality algebra is a symmetric equality algebra.

\begin{rems} \label{com-eq-60} $\rm($\cite{Ciu9}$\rm)$ 
$(1)$ A pseudo equality algebra $(A, \wedge, \thicksim, \backsim, 1)$ is commutative 
if and only if its corresponding pseudo BCK(pC)-meet-semilattice $\Psi(A)$ is commutative. \\
$(2)$ Every finite invariant commutative pseudo equality algebra is a symmetric pseudo equality algebra. \\
$(3)$ If $(A, \wedge, \thicksim, \backsim, 1)$ is a symmetric pseudo equality algebra, then 
$\Psi(A)=(A, \wedge, \rightarrow, \rightsquigarrow, 1)$ is a BCK(pC)-meet-semilattice. 
\end{rems}

In what follows we recall some notions and results regarding the deductive systems and congruences on a 
pseudo equality algebra (see \cite{Dvu7}). \\
Let $(A, \wedge, \thicksim, \backsim, 1)$ be a pseudo equality algebra. A subset $D\subseteq A$ is called a \emph{deductive system} of $A$ if for all $x, y\in A$: \\
$(DS_1)$ $1 \in D;$ \\
$(DS_2)$ if $x\in D$, $x\le y$, then $y\in D;$ \\
$(DS_3)$ if $x, y \thicksim x \in D$, then $y\in D$. \\
A subset $D\subseteq A$ is a deductive system of $A$ if, for all $x, y\in A$, it satisfies conditions $(DS_1)$, $(DS_2)$ and the condition:\\
$(DS_3^{\prime})$ if $x, x \backsim y \in D$, then $y\in D$. \\
A deductive system $D$ of a pseudo equality algebra $A$ is \emph{proper} if $D\ne A$. 
A proper deductive system is called \emph{maximal} if it is not strictly contained in any other proper 
deductive system of $A$. 
We will denote by ${\mathcal DS}(A)$ the set of all deductive systems of $A$. 
Clearly, $\{1\}, A \subseteq {\mathcal DS}(A)$ and ${\mathcal DS}(A)$ is closed under arbitrary intersections. 
As a consequence, $({\mathcal DS}(A), \subseteq)$ is a complete lattice. 
The set of deductive systems of an invariant  pseudo equality algebra coincides with the set of deductive systems of its corresponding pseudo BCK(pC)-meet-semilattice. \\
A deductive system $D$ of $A$ is called \emph{closed} if $x\thicksim y, x\backsim y\in D$ for all $x, y\in D$. 
According to \cite[Prop. 4.5]{Dvu7}, a deductive system $D$ of a pseudo equality algebra $A$ is closed 
if and only if $1\thicksim x, x\backsim 1\in D$ for all $x\in D$. 

\begin{prop} \label{lds-eqa-45} Every deductive system of an invariant pseudo equality algebra $A$ is a subalgebra of $A$.
\end{prop}
\begin{proof} 
Let $D\in {\mathcal DS}(A)$, so by $(DS_1)$, $1\in D$. Consider $x, y\in D$. \\ 
According to \cite[Ex. 4.6]{Dvu7}, $D$ is closed, that is $x\thicksim y, x\backsim y\in D$ for all $x, y\in D$.
Since by Proposition \ref{ps-eq-20}$(5)$ $x\thicksim y \le x\wedge y\thicksim y$, we get $x\wedge y\thicksim y \in D$ and finally, from $y, x\wedge y\thicksim y \in D$ it follows that $x\wedge y \in D$. 
Thus $D$ is a subalgebra of $A$.
\end{proof}

A deductive system $D$ of a pseudo equality algebra $A$ is called \emph{normal} if it satisfies the condition:\\
$(DS_4)$ $x \thicksim y, y\thicksim x \in D$ iff $y \backsim x, x\backsim y \in D$, for all $x, y\in A$. 
We will denote by ${\mathcal DS}_n(A)$ the set of all normal deductive systems of $A$. \\
Obviously $\{1\}, A \in {\mathcal DS}_n(A)$ and ${\mathcal DS}_n(A)\subseteq {\mathcal DS}(A)$.
A subset $\Theta \subseteq A \times A$ is called a \emph{congruence} of $A$ if it is an 
equivalence relation on $A$ and for all $x_1, y_1, x_2, y_2\in A$ such that $(x_1, y_1), (x_2, y_2) \in \Theta$ 
the following hold: \\
$(CG_1)$ $(x_1 \wedge x_2, y_1 \wedge y_2) \in \Theta;$ \\
$(CG_2)$ $(x_1 \thicksim x_2, y_1 \thicksim y_2) \in \Theta;$ \\
$(CG_2)$ $(x_1 \backsim x_2, y_1 \backsim y_2) \in \Theta$. \\
We will denote by ${\mathcal Con}(A)$ the set of all congruences of $A$. \\
With any $H\in {\mathcal DS}_n(A)$ we associate a binary relation $\Theta_H$ by defining $x\Theta_H y$ iff 
$x \thicksim y \in H$ iff $x \backsim y \in H$. \\
If $\Theta$ is congruence relation on a pseudo 
equality algebra $(A, \wedge, \thicksim, \backsim, 1)$, then 
$F_{\Theta}=[1]_{\Theta}=\{x\in A\mid (x, 1)\in \Theta \}$ is a closed normal deductive system of $A$. \\
If $H\in {\mathcal DS}_n(A)$, then  
$H_{\Theta}=\{(x, y)\in A \times A \mid x\thicksim y, y\thicksim x \in H\}=
\{(x, y)\in A \times A \mid x\backsim y, y\backsim x \in H\} \in {\mathcal Con}(A)$.

\begin{theo} \label{lds-eqa-100} Let $(A, \wedge, \thicksim, \backsim, 1)$ be an invariant pseudo equality algebra. 
Then there is a one-to-one correspondence between the set of all normal deductive systems of $A$ and 
${\mathcal Con}(A)$. 
\end{theo}
\begin{proof}
According to \cite[Th. 4.11]{Dvu7}, there is a one-to-one correspondence between the set of all normal closed deductive systems of $A$ and ${\mathcal Con}(A)$. The assertion follows from the fact that any deductive system of an invariant pseudo equality algebra is closed (\cite[Ex. 4.6]{Dvu7}).
\end{proof}

Let $(A, \wedge, \thicksim, \backsim, 1)$ be a pseudo equality algebra and $H\in {\mathcal DS}_n(A)$. \\
Denote $A/{\Theta_H}=\{x/{\Theta_H}\mid x\in A\}$, where $x/{\Theta_H}=\{y\in A\mid (x,y)\in \Theta_H\}$. 
We define the following operations on $A/{\Theta_H}$: 
$x/{\Theta_H} \wedge y/{\Theta_H}=(x\wedge y)/{\Theta_H}$, 
$x/{\Theta_H} \thicksim y/{\Theta_H}=(x\thicksim y)/{\Theta_H}$, 
$x/{\Theta_H} \backsim y/{\Theta_H}=(x\backsim y)/{\Theta_H}$. \\
If $H\in {\mathcal DS}_n(A)$, then $(A/{\Theta_H},\wedge,\thicksim,\backsim,1/{\Theta_H})$ is a pseudo equality algebra. \\ 
A pseudo equality algebra $(A, \wedge, \thicksim, \backsim, 1)$ is called \emph{simple} if 
${\mathcal DS}(A)=\{\{1\}, A\}$.

$\vspace*{5mm}$

\section{Examples of pseudo equality algebras}

In this section we give examples of pseudo equality algebras and their classes - invariant, commutative 
and symmetric pseudo equality algebras. 

\begin{ex} \label{ps-ex-10} $\rm($\cite{Dvu7}$\rm)$
Let $(G, \vee,\wedge, \cdot, ^{-1}, e)$ be an $\ell$-group. On the negative cone $G^{-}=\{g\in G \mid g\le e\}$ we define the operations 
$x\thicksim y=(x\cdot y^{-1})\wedge e$, $x\backsim y=(x^{-1}\cdot y)\wedge e$.  
Then $(G^{-}, \wedge, \thicksim, \backsim, e)$ is a pseudo equality algebra. 
We have $x\thicksim y=y\backsim x$ if and only if $G$ is Abelian.
\end{ex}

\begin{ex} \label{ps-ex-20} $\rm($\cite{Dvu7}$\rm)$
Let $(A, \odot, \rightarrow, \rightsquigarrow, 1)$ be a pseudo-hoop. 
If we define the operations $x\thicksim y=y\rightarrow x$, $x\backsim y=x\rightsquigarrow y$, 
then $(A, \wedge, \thicksim, \backsim, 1)$ is a pseudo equality algebra, where 
$x\wedge y=(x\rightarrow y)\odot x=(y\rightarrow x)\odot y=x\odot(x\rightsquigarrow y)=y\odot(y\rightsquigarrow x)$.
\end{ex}

\begin{ex} \label{ps-ex-30} 
Let $(A, \wedge, \rightarrow, \rightsquigarrow, 1)$ be a pseudo BCK(pD)-meet-semilattice. 
Define the operations $x\thicksim y=y\rightarrow x$, $x\backsim y=x\rightsquigarrow y$. 
According to Proposition \ref{ps-eq-100}, $A$ is a pseudo BCK(pC)-meet-semilattice, and by Theorem \ref{ps-eq-110},   
$(A, \wedge, \thicksim, \backsim, 1)$ is a pseudo equality algebra.
\end{ex}

\begin{ex} \label{ps-ex-40} 
Let $(A, \wedge, \rightarrow, \rightsquigarrow, 1)$ be a linearly ordered pseudo BCK-meet-semilattice. 
By Remark \ref{ps-eq-80} and Proposition \ref{ps-eq-100}, $A$ is a pseudo BCK(pC)-meet-semilattice, thus by Theorem \ref{ps-eq-110}, $(A, \wedge, \thicksim, \backsim, 1)$ is a pseudo equality algebra, where 
$x\thicksim y=y\rightarrow x$, $x\backsim y=x\rightsquigarrow y$.
\end{ex}

\begin{ex}\label{ps-ex-50} 
Let $B=\{0,a,b,1\}$ with $0<a, b<1$ be a lattice whose diagram is below. 
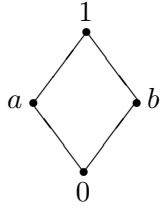
\begin{center}
\begin{picture}(50,100)(0,0)
\put(37,11){\circle*{3}}
\put(34,0){$0$}
\put(37,11){\line(3,4){20}}
\put(57,37){\circle*{3}}
\put(61,35){$b$}

\put(37,11){\line(-3,4){20}}
\put(18,37){\circle*{3}}
\put(8,35){$a$}

\put(18,37){\line(3,4){20}}
\put(38,64){\circle*{3}}
\put(35,68){$1$} 

\put(38,64){\line(3,-4){20}}

\end{picture}
\end{center}

Then $(B, \wedge, \rightarrow, 1)$ is a BCK(P)-lattice with the operations $\rightarrow$, $\odot$ given by the 
tables below:  
\[
\hspace{10mm}
\begin{array}{c|ccccc}
\rightarrow& 0 & a & b & 1 \\ \hline
0 & 1 & 1 & 1 & 1 \\ 
a & b & 1 & b & 1 \\ 
b & a & a & 1 & 1 \\  
1 & 0 & a & b & 1
\end{array}
\hspace{10mm}
\begin{array}{c|ccccc}
\odot& 0 & a & b & 1 \\ \hline
0 & 0 & 0 & 0 & 0 \\ 
a & 0 & a & 0 & a \\ 
b & 0 & 0 & b & b \\  
1 & 0 & a & b & 1
\end{array}
.
\]

Thus $(B, \wedge, \rightarrow, 1)$ is a BCK(D)-lattice, so it is a BCK(C)-lattice, and 
$\Phi(B)=(B, \wedge, \thicksim, \backsim, 1)$ is a pseudo equality algebra with the operations 
$\thicksim, \backsim$ given below: 
\[
\hspace{10mm}
\begin{array}{c|ccccc}
\thicksim& 0 & a & b & 1 \\ \hline
0 & 1 & b & a & 0 \\ 
a & 1 & 1 & a & a \\ 
b & 1 & b & 1 & b \\  
1 & 1 & 1 & 1 & 1
\end{array}
\hspace{10mm}
\begin{array}{c|ccccc}
\backsim& 0 & a & b & 1 \\ \hline
0 & 1 & 1 & 1 & 1 \\ 
a & b & 1 & b & 1 \\ 
b & a & a & 1 & 1 \\  
1 & 0 & a & b & 1
\end{array}
.
\]
One can easily chack that $\Phi(\Psi(B))=B$, thus $(B, \wedge, \thicksim, \backsim, 1)$ is an invariant pseudo 
equality algebra. We mention that ${\mathcal DS}(B)={\mathcal DS}_n(B)=\{\{1\}, \{a,1\}, \{b,1\}, B\}$. 
\end{ex}

\begin{ex}\label{ps-ex-60} 
Consider the lattice $A=\{0,a,b,1\}$ with $0<a, b<1$ whose diagram is given in Example \ref{ps-ex-50}.   
Then the structure $(A, \wedge, \thicksim, 1)$ is an equality algebra with $\thicksim$ given below:
\[
\hspace{10mm}
\begin{array}{c|ccccc}
\thicksim& 0 & a & b & 1 \\ \hline
0 & 1 & b & a & 0 \\ 
a & b & 1 & 0 & a \\ 
b & a & 0 & 1 & b \\  
1 & 0 & a & b & 1
\end{array}
.
\]
If $(B, \wedge, \rightarrow, 1)$ is the BCK(C)-lattice from Example \ref{ps-ex-50}, we can see that $\Psi(A)=B$, 
but $\Phi(\Psi(A))\neq A$, hence $(A, \wedge, \thicksim, 1)$ is not invariant.
\end{ex}

\begin{ex}\label{ps-ex-70} Let $(A_1, \wedge_1, \thicksim_1, \backsim_1, 1_1)$ and 
$(A_2, \wedge_2, \thicksim_2, \backsim_2, 1_2)$ be two pseudo equality algebras.
Denote $A=A_1 \times A_2=\{(x_1,x_2) \mid x_1\in A_1, x_2\in A_2\}$ and, for all $(x_1, x_2), (y_1, y_2)\in A$ 
define the operations $\wedge, \thicksim, \backsim, 1$ as follows: 
$(x_1, x_2)\wedge (y_1, y_2)=(x_1\wedge_1 y_1, x_2\wedge_2 y_2)$, 
$(x_1, x_2)\thicksim (y_1, y_2)=(x_1\thicksim_1 y_1, x_2\thicksim_2 y_2)$, 
$(x_1, x_2)\backsim (y_1, y_2)=(x_1\backsim_1 y_1, x_2\backsim_2 y_2)$, 
$1=(1_1, 1_2)$. 
Then $(A, \wedge, \thicksim, \backsim, 1)$ is a pseudo equality algebra. Moreover: \\
$(1)$ If $D_1\in {\mathcal DS}(A_1)$, $D_2\in {\mathcal DS}(A_2)$, then $D_1\times D_2\in {\mathcal DS}(A);$ \\
$(2)$ If $D_1\in {\mathcal DS}_n(A_1)$, $D_2\in {\mathcal DS}_n(A_2)$, then $D_1\times D_2\in {\mathcal DS}_n(A)$.
\end{ex} 

\begin{ex} \label{ps-ex-80} 
Let $(A, \wedge, \rightarrow, \rightsquigarrow, 1)$ be a commutative pseudo BCK-meet-semilattice, that is  
$(x\rightarrow y)\rightsquigarrow y=(y\rightarrow x)\rightsquigarrow x$ and 
$(x\rightsquigarrow y)\rightarrow y=(y\rightsquigarrow x)\rightarrow x$ for all $x, y\in A$. \\
By \cite[Lemma 4.1.12]{Kuhr6}, $A$ is a pseudo BCK(pD)-meet-semilattice, so $(A, \wedge, \thicksim, \backsim, 1)$ is a pseudo equality algebra, where $x\thicksim y=y\rightarrow x$, $x\backsim y=x\rightsquigarrow y$.
\end{ex}

\begin{ex} \label{ps-ex-90}
Let $(G, \vee,\wedge, \cdot, ^{-1}, e)$ be an $\ell$-group. On the negative cone $G^{-}=\{g\in G \mid g\le e\}$ we define the operations $x\rightarrow y=y\cdot (x\vee y)^{-1}$, $x\rightsquigarrow y=(x\vee y)^{-1}\cdot y$. 
According to \cite[Example 4.1.2]{Kuhr6} and Example \ref{ps-ex-80} , 
$(G^{-}, \wedge, \rightarrow, \rightsquigarrow, e)$ is a commutative pseudo BCK(pD)-meet-semilattice. 
By Remarks \ref{com-eq-60} it follows that $\Phi(G^{-})$ is a commutative pseudo equality algebra.
\end{ex}

\begin{ex}\label{ps-ex-100} 
The pseudo equality algebra $(B, \wedge, \thicksim, \backsim, 1)$ from Example \ref{ps-ex-50} is commutative 
and symmetric.
\end{ex}

$\vspace*{5mm}$

\section{Pointed pseudo equality algebras}

We define and investigate the pointed pseudo equality algebras and compatible pseudo equality algebras, and we show that a good pseudo-hoop is a compatible pseudo equality algebra.

\begin{Def} \label{pps-eq-10}
A pseudo equality algebra $\mathcal{A}=(A, \wedge, \thicksim, \backsim, 1)$ with a constant $a\in A$ (which 
can denote any element) is called a \emph{pointed pseudo equality algebra}.
\end{Def} 

A pointed pseudo equality algebra is denoted by $\mathcal{A}=(A, \wedge, \thicksim, \backsim, a, 1)$. 

\begin{Def} \label{pps-eq-20} Let $\mathcal{A}=(A, \wedge, \thicksim, \backsim, a, 1)$ be a pointed pseudo equality algebra. For every $x\in A$ define two pairs of \emph{negations relative to $a$} or \emph{$a$-relative negations}: \\ 
$\hspace*{3cm}$ $x^{\thicksim_a}=a\thicksim x, \:\:$   
                $x^{\backsim_a}=x\backsim a$ and \\ 
$\hspace*{3cm}$ $x^{\rightarrow_a}=x\rightarrow a, \:\:$    
                $x^{\rightsquigarrow_a}=x\rightsquigarrow a$.  
\end{Def}

\begin{prop} \label{pps-eq-30} In any pointed pseudo equality algebra $(A, \wedge, \thicksim, \backsim, a, 1)$ the 
following hold for all $x, y\in A$:\\
$(1)$ $1^{\thicksim_a}=1^{\backsim_a}=a$ and $1^{\thicksim_a\backsim_a}=1^{\backsim_a\thicksim_a}=1;$ \\
$(2)$ $a^{\thicksim_a}=a^{\backsim_a}=1$ and $a^{\thicksim_a\backsim_a}=a^{\backsim_a\thicksim_a}=a;$ \\
$(3)$ $x\le x^{\thicksim_a \backsim_a}$ and $x\le x^{\backsim_a\thicksim_a};$ \\
$(4)$ $x\thicksim y\le x^{\thicksim_a} \backsim y^{\thicksim_a}\le 
                       x^{\thicksim_a\backsim_a} \thicksim y^{\thicksim_a\backsim_a}$ and \\
$\hspace*{0.5cm}$                       
      $x\backsim y\le x^{\backsim_a} \thicksim y^{\backsim_a} \le 
                      x^{\backsim_a\thicksim_a} \backsim y^{\backsim_a\thicksim_a}$.                     
\end{prop}      
\begin{proof}
$(3)$ It follows from Proposition \ref{ps-eq-10}$(2)$ for $y=a$. \\
$(4)$ Applying $(A_6)$ for $z=a$ we get: \\
$\hspace*{2cm}$
      $x\thicksim y\le x^{\thicksim_a} \backsim y^{\thicksim_a}$ and 
      $x\backsim y\le x^{\backsim_a} \thicksim y^{\backsim_a}$. \\
Replacing $x$ with $x^{\thicksim_a}$ and $y$ with $y^{\thicksim_a}$ in the second inequality, it follows that: \\
$\hspace*{2cm}$
$x^{\thicksim_a} \backsim y^{\thicksim_a}\le x^{\thicksim_a\backsim_a} \thicksim y^{\thicksim_a\backsim_a}$. \\
Replacing $x$ with $x^{\backsim_a}$ and $y$ with $y^{\backsim_a}$ in the first inequality, we have: \\ 
$\hspace*{2cm}$
$x^{\backsim_a} \thicksim y^{\backsim_a} \le x^{\backsim_a\thicksim_a} \backsim y^{\backsim_a\thicksim_a}$. \\
We conclude that:\\
$\hspace*{2cm}$
      $x\thicksim y\le x^{\thicksim_a} \backsim y^{\thicksim_a}\le 
                       x^{\thicksim_a\backsim_a} \thicksim y^{\thicksim_a\backsim_a}$ and \\
$\hspace*{2cm}$                       
      $x\backsim y\le x^{\backsim_a} \thicksim y^{\backsim_a} \le 
                      x^{\backsim_a\thicksim_a} \backsim y^{\backsim_a\thicksim_a}$. 
\end{proof}      

\begin{prop} \label{pps-eq-30-10} Let $(A, \wedge, \thicksim, \backsim, 0, 1)$ be a bounded pseudo equality algebra.
Then for all $x\in A$ we have: $x^{\thicksim_0\backsim_0\thicksim_0}=x^{\thicksim_0}$ and 
$x^{\backsim_0\thicksim_0\backsim_0}=x^{\backsim_0}$.
\end{prop}
\begin{proof}
It is a consequence of Proposition \ref{ps-eq-30}, since $0\le x$ for all $x\in A$.
\end{proof}

\begin{prop} \label{pps-eq-40} In any pointed pseudo equality algebra $(A, \wedge, \thicksim, \backsim, a, 1)$ the 
following hold for all $x, y\in A$:\\
$(1)$ $1^{\rightarrow_a}=1^{\rightsquigarrow_a}=a;$ \\
$(2)$ $1^{\rightarrow_a\rightsquigarrow_a}=1^{\rightsquigarrow_a\rightarrow_a}=1;$ \\
$(3)$ $a^{\rightarrow_a}=a^{\rightsquigarrow_a}=1;$ \\ 
$(4)$ $a^{\rightarrow_a\rightsquigarrow_a}=a^{\rightsquigarrow_a\rightarrow_a}=a;$ \\
$(5)$ $x\le x^{\rightarrow_a\rightsquigarrow_a}$ and $x\le x^{\rightsquigarrow_a\rightarrow_a};$ \\
$(6)$ $a\le x^{\rightarrow_a\rightsquigarrow_a}$ and $a\le x^{\rightsquigarrow_a\rightarrow_a};$ \\
$(7)$ $x^{\rightarrow_a\rightsquigarrow_a\rightarrow_a}=x^{\rightarrow_a}$ and 
      $x^{\rightsquigarrow_a\rightarrow_a\rightsquigarrow_a}=x^{\rightsquigarrow_a};$ \\
$(8)$ $x\rightarrow y\le y^{\rightarrow_a}\rightsquigarrow x^{\rightarrow_a} \le 
       x^{\rightarrow_a\rightsquigarrow_a} \rightarrow y^{\rightarrow_a\rightsquigarrow_a}$ and \\
   $\hspace*{0.5cm}$    
      $x\rightsquigarrow y\le y^{\rightsquigarrow_a}\rightarrow x^{\rightsquigarrow_a} \le 
       x^{\rightsquigarrow_a\rightarrow_a} \rightsquigarrow ^{\rightsquigarrow_a\rightarrow_a};$ \\
$(9)$ $x^{\thicksim_a}\le x^{\rightarrow_a},\: x^{\rightsquigarrow_a}$ and 
      $x^{\backsim_a}\le x^{\rightarrow_a},\: x^{\rightsquigarrow_a};$ \\
$(10)$ $x^{\thicksim_a}=x^{\rightarrow_a}$ and $x^{\backsim_a}=x^{\rightsquigarrow_a}$ for all $x\ge a$.             
\end{prop}
\begin{proof} 
$(5)$ It follows from Proposition \ref{ps-eq-10}$(8)$ for $y=a$. \\
$(7)$ It follows from the identities 
      $((x\rightarrow a)\rightsquigarrow a)\rightarrow a =x\rightarrow a$ and 
      $((x\rightsquigarrow a)\rightarrow a)\rightsquigarrow a =x\rightsquigarrow a$ 
which hold in any pseudo BCK-algebra (see for example \cite{Ciu2}). \\ 
$(8)$ Applying Proposition \ref{ps-eq-10}$(9)$ for $z=a$ we have 
$x\rightarrow y\le y^{\rightarrow_a}\rightsquigarrow x^{\rightarrow_a}$ and 
$x\rightsquigarrow y\le y^{\rightsquigarrow_a}\rightarrow x^{\rightsquigarrow_a}$. 
By the second inequality we get $y^{\rightarrow_a}\rightsquigarrow x^{\rightarrow_a} \le 
x^{\rightarrow_a\rightsquigarrow_a} \rightarrow y^{\rightarrow_a\rightsquigarrow_a}$ and 
by the first one we get 
$y^{\rightsquigarrow_a}\rightarrow x^{\rightsquigarrow_a} \le 
x^{\rightsquigarrow_a\rightarrow_a} \rightsquigarrow ^{\rightsquigarrow_a\rightarrow_a}$. \\
$(9)$ It follows from Proposition \ref{ps-eq-10}$(1)$. \\
$(10)$ Since $a\le x$, we have $x^{\thicksim_a}=a\thicksim x=x\wedge a\thicksim x=x\rightarrow a=x^{\rightarrow_a}$ 
and $x^{\backsim_a}=x\backsim a = x\backsim x\wedge a=x\rightsquigarrow a=x^{\rightsquigarrow_a}$.
\end{proof}

\begin{Def} \label{pps-eq-60} A pointed pseudo equality algebra $(A, \wedge, \thicksim, \backsim, a, 1)$ is said 
to be: \\
$(i)$  \emph{$(^{\thicksim_a, \backsim_a})$-good}, if $x^{\thicksim_a\backsim_a}=x^{\backsim_a\thicksim_a}$ 
for all $x\in A;$ \\
$(ii)$ \emph{$(^{\rightarrow_a, \rightsquigarrow_a})$-good}, if $x^{\rightarrow_a\rightsquigarrow_a}=x^{\rightsquigarrow_a\rightarrow_a}$ for all $x\in A$.
\end{Def}

\begin{Def} \label{pps-eq-70} A pointed pseudo equality algebra $(A, \wedge, \thicksim, \backsim, a, 1)$ is said 
to be: \\
$(i)$  \emph{$(^{\thicksim_a, \backsim_a})$-involutive}, if $x^{\thicksim_a\backsim_a}=x^{\backsim_a\thicksim_a}=x$ 
for all $x\in A;$ \\
$(ii)$ \emph{$(^{\rightarrow_a, \rightsquigarrow_a})$-involutive}, if $x^{\rightarrow_a\rightsquigarrow_a}=x^{\rightsquigarrow_a\rightarrow_a}=x$ for all $x\in A$.
\end{Def}

\begin{rem} \label{pps-eq-80} Any pointed pseudo equality algebra $(A, \wedge, \thicksim, \backsim, a, 1)$ is 
$(^{\thicksim_1, \backsim_1})$-involutive and $(^{\rightarrow_1, \rightsquigarrow_1})$-good.
\end{rem}


\begin{prop} \label{pps-eq-100} Let $(A, \wedge, \thicksim, \backsim, a, 1)$ be a pointed pseudo equality algebra. 
Then the following hold for all $x, y\in A$: \\  
$(1)$ if $A$ is $(^{\thicksim_a, \backsim_a})$-involutive, then 
      $x\thicksim y=x^{\thicksim_a}\backsim y^{\thicksim_a}$ and 
      $x\thicksim y=x^{\backsim_a}\thicksim y^{\backsim_a};$ \\
$(2)$ if $A$ is $(^{\rightarrow_a, \rightsquigarrow_a})$-involutive, then 
      $x\rightarrow y= y^{\rightarrow_a}\rightsquigarrow x^{\rightarrow_a}$ and 
      $x\rightsquigarrow y=y^{\rightsquigarrow_a}\rightarrow x^{\rightsquigarrow_a}$. 
\end{prop}
\begin{proof}
It follows by Propositions \ref{pps-eq-30}$(4)$ and \ref {pps-eq-40}$(8)$, respectively. 
\end{proof}

\begin{prop} \label{pps-eq-110} Let $(A, \wedge, \thicksim, \backsim, a, 1)$ be a pointed pseudo equality algebra.  
Consider the terms $F_1(x,y,z)=(x\thicksim y) \backsim z$ and $F_2(x,y,z)=z\thicksim (y \backsim x)$ satisfying the 
conditions $F_1(x\wedge a,x,a)=F_1(a,x,a)$ and $F_2(x\wedge a,x,a)=F_2(a,x,a)$ 
for all $x\in A$. Then $A$ is $(^{\thicksim_a, \backsim_a})$-good iff it is 
$(^{\rightarrow_a, \rightsquigarrow_a})$-good.  
\end{prop}
\begin{proof} For all $x\in A$ we have: \\
$\hspace*{3cm}$ $x^{\thicksim_a \backsim_a}=(a\thicksim x)\backsim a$, \\
$\hspace*{3cm}$ $x^{\backsim_a \thicksim_a}= a\thicksim (x\backsim a)$, \\
$\hspace*{3cm}$ $x^{\rightarrow_a\rightsquigarrow_a}=(x\wedge a\thicksim x)\backsim a$, \\
$\hspace*{3cm}$ $x^{\rightsquigarrow_a\rightarrow_a}=a\thicksim (x\backsim x\wedge a)$. \\
It follows that $A$ is $(^{\thicksim_a, \backsim_a})$-good iff  
$(a\thicksim x)\backsim a=a\thicksim (x\backsim a)$ for all $x\in A$ iff $F_1(a,x,a)=F_2(a,x,a)$ for all $x\in A$ iff 
$F_1(x\wedge a,x,a)=F_2(x\wedge a,x,a)$ for all $x\in A$ iff $x^{\rightarrow_a\rightsquigarrow_a}=x^{\rightsquigarrow_a\rightarrow_a}$ 
for all $x\in A$ iff $A$ is $(^{\rightarrow_a, \rightsquigarrow_a})$-good. 
\end{proof}

\begin{cor} \label{pps-eq-110-10}
A  bounded pseudo equality algebra $(A, \wedge, \thicksim, \backsim, 0, 1)$ is  
$(^{\thicksim_0, \backsim_0})$-good iff it is $(^{\rightarrow_0, \rightsquigarrow_0})$-good.
\end{cor}

\begin{Def} \label{pps-eq-120} A $(^{\thicksim_a, \backsim_a})$-good pseudo equality algebra 
$(A, \wedge, \thicksim, \backsim, a, 1)$ is said to be \emph{compatible with respect to $a$} or 
\emph{$a$-compatible} if it satisfies the following conditions for all $x, y\in A$: \\
$(C_1)$ $(x\thicksim y)^{\thicksim_a\backsim_a}=x^{\thicksim_a\backsim_a}\thicksim y^{\thicksim_a\backsim_a};$ \\
$(C_2)$ $(x\backsim y)^{\thicksim_a\backsim_a}=x^{\thicksim_a\backsim_a}\backsim y^{\thicksim_a\backsim_a};$ \\
$(C_3)$ $(x\wedge y)^{\thicksim_a\backsim_a}=x^{\thicksim_a\backsim_a}\wedge y^{\thicksim_a\backsim_a};$ \\
$(C_4)$ $x^{\thicksim_a\backsim_a\thicksim_a}=x^{\thicksim_a}$ and 
        $x^{\backsim_a\thicksim_a\backsim_a}=x^{\backsim_a}$.  
\end{Def}

\begin{ex} \label{pps-eq-130}
Any $(^{\thicksim_a, \backsim_a})$-involutive pseudo equality algebra $(A, \wedge, \thicksim, \backsim, a, 1)$ is $a$-compatible. 
\end{ex}

\begin{ex} \label{pps-eq-130-10} 
Let $\mathcal{B}=(B, \odot, \rightarrow, \rightsquigarrow, 0, 1)$ be a good pseudo-hoop.
According to \cite{Ciu2}, the following hold for all $x, y\in B$: \\
$(1)$ $(x \rightarrow y)^{\rightarrow_0\rightsquigarrow_0}=
        x^{\rightarrow_0\rightsquigarrow_0} \rightarrow y^{\rightarrow_0\rightsquigarrow_0};$ \\
$(2)$ $(x \rightsquigarrow y)^{\rightarrow_0\rightsquigarrow_0}=
        x^{\rightarrow_0\rightsquigarrow_0} \rightsquigarrow y^{\rightarrow_0\rightsquigarrow_0};$ \\
$(3)$ $(x\wedge y)^{\rightarrow_0\rightsquigarrow_0}=
        x^{\rightarrow_0\rightsquigarrow_0}\wedge y^{\rightarrow_0\rightsquigarrow_0};$ \\
$(4)$ $x^{\rightarrow_0\rightsquigarrow_0\rightarrow_0}=x^{\rightarrow_0}$ and 
      $x^{\rightsquigarrow_0\rightarrow_0\rightsquigarrow_0}=x^{\rightsquigarrow_0}$. \\  
According to Remark \ref{ps-eq-80}, $B$ is a pseudo BCK(pD)-meet-semilattice and applying Theorem \ref{ps-eq-110} 
it follows that the structure $\Phi(\mathcal{B})=(B, \wedge, \thicksim, \backsim, 0, 1)$ 
is a pointed pseudo equality algebra, where $x\thicksim y=y\rightarrow x$ and $x\backsim y=x\rightsquigarrow y$.\\ 
Obviously from the goodness property of $\mathcal{B}$ it follows that $\Phi(\mathcal{B})$ is $0$-good. \\
Since by Proposition \ref{pps-eq-40}$(10)$, $\rightarrow_0=\thicksim_0$ and $\rightsquigarrow_0=\backsim_0$ we get: \\
$\hspace*{1cm}$
$(x\thicksim y)^{\thicksim_0\backsim_0}=(y\rightarrow x)^{\rightarrow_0\rightsquigarrow_0}=
y^{\rightarrow_0\rightsquigarrow_0} \rightarrow x^{\rightarrow_0\rightsquigarrow_0}= 
x^{\thicksim_0\backsim_0} \thicksim y^{\thicksim_0\backsim_0}$, \\
$\hspace*{1cm}$
$(x\backsim y)^{\thicksim_0\backsim_0}=(x\rightarrow y)^{\rightarrow_0\rightsquigarrow_0}=
x^{\rightarrow_0\rightsquigarrow_0} \rightarrow y^{\rightarrow_0\rightsquigarrow_0}= 
x^{\thicksim_0\backsim_0} \backsim y^{\thicksim_0\backsim_0}$, \\
for all $x, y\in B$. \\ 
Moreover, from $(3)$ and $(4)$ we get 
$(x\wedge y)^{\thicksim_0\backsim_0}=
        x^{\thicksim_0\backsim_0}\wedge y^{\thicksim_0\backsim_0}$,  
$x^{\thicksim_0\backsim_0\thicksim_0}=x^{\thicksim_0}$ and 
      $x^{\backsim_0\thicksim_0\backsim_0}=x^{\backsim_0}$,  
thus $\Phi(\mathcal{B})$ is a $0$-compatible pseudo equality algebra.           
\end{ex}

\begin{prop} \label{pps-eq-130-20} Let $(A, \wedge, \thicksim, \backsim, a, 1)$ be an $a$-compatible  pseudo equality  algebra. Then the map $\gamma:A\longrightarrow A$ defined by $\gamma(x)=x^{\thicksim_a\backsim_a}$ for all $x\in A$ 
is a closure operator on $A$. 
\end{prop}
\begin{proof} We have to prove that $\gamma$ satisfies the following conditions for all $x, y\in A$: \\
$(i)$ $x\le \gamma(x)$ (extensive); \\
$(ii)$ $x\le y$ implies $\gamma(x)\le \gamma(y)$ (monotone); \\
$(iii)$ $\gamma(\gamma(x))=\gamma(x)$ (idempotent). \\
Indeed, condition $(i)$ follows from Proposition \ref{pps-eq-30}$(3)$. \\ 
Since $x\le y$, we have $x=x\wedge y$, and applying $(C_3)$ we get: 
$x^{\thicksim_a\backsim_a}=x^{\thicksim_a\backsim_a}\wedge y^{\thicksim_a\backsim_a}$, that is 
$x^{\thicksim_a\backsim_a}\le y^{\thicksim_a\backsim_a}$. It follows that $\gamma(x)\le \gamma(y)$, hence 
$\gamma$ verifies condition $(ii)$. \\
Applying $(C_4)$ we get: 
$\gamma(\gamma(x))=x^{\thicksim_a\backsim_a\thicksim_a\backsim_a}=x^{\thicksim_a\backsim_a}=\gamma(x)$, 
that is $(iii)$. \\
We conclude that $\gamma$ is a closure operator on $A$. 
\end{proof}

\begin{Def} \label{pps-eq-140} Let $(A, \wedge, \thicksim, \backsim, a, 1)$ be a pointed pseudo equality algebra. 
An element $x\in A$ is said to be \emph{$a$-regular} if 
$x^{\thicksim_a\backsim_a}=x^{\backsim_a\thicksim_a}=x$. \\
Denote $\Reg_a(A)$ the set of all $a$-regular elements of $A$.
\end{Def}

\begin{rem} \label{pps-eq-150} Let $(A, \wedge, \thicksim, \backsim, a, 1)$ be a pointed pseudo equality algebra. 
Then: \\
$(1)$ $1, a\in \Reg_a(A)$. \\
$(2)$ $A$ is $^{\thicksim_a\backsim_a}$-involutive iff $A=\Reg_a(A);$ \\
$(3)$ if $A$ is $a$-compatible, then $\Reg_a(A)$ is a subalgebra of $A$. 
\end{rem}

\begin{ex} \label{pps-eq-160} Consider the pseudo equality algebra $(B, \wedge, \thicksim, \backsim, 1)$ 
from Example \ref{ps-ex-50}. \\
We have: $\Reg_0(B)=\{0\}$, $\Reg_a(B)=\{a\}$, $\Reg_b(B)=\{b\}$, $\Reg_1(B)=\{1\}$.
\end{ex}

$\vspace*{5mm}$

\section{Bosbach states on pointed pseudo equality algebras}

In this section we introduce the notion of a Bosbach state on a pointed pseudo equality algebra and we prove that 
any Bosbach state on a pointed pseudo equality algebra is also a Bosbach state on its corresponding pointed 
pseudo BCK(pC)-meet-semilattice. For the case of an invariant pointed pseudo equality algebra, we show that 
the Bosbach states on the two structures coincide. 

\begin{Def} \label{bs-eq-10} Let $(A, \wedge, \thicksim, \backsim, a, 1)$ be a pointed pseudo equality algebra 
with $a\ne 1$.  
A \emph{Bosbach state} on $A$ is a function $s:A\longrightarrow [0,1]$ satisfying the following axioms: \\
$(BS_1)$ $s(x)+s(x\wedge y\thicksim x)=s(y)+s(x\wedge y\thicksim y),$ \\
$(BS_2)$ $s(x)+s(x\backsim x\wedge y)=s(y)+s(y\backsim x\wedge y),$ \\
$(BS_3)$ $s(1)=1$ and $s(a)=0$, \\
for all $x, y\in A$.
\end{Def}

Denote $\mathcal{BS}^{(a)}_{EQA}(A)$ the set of all Bosbach states on a pointed pseudo equality algebra 
$(A, \wedge, \thicksim, \backsim, a, 1)$. 

\begin{prop} \label{bs-eq-20} Let $(A, \wedge, \thicksim, \backsim, a, 1)$ be a pointed pseudo equality algebra and 
$s\in \mathcal{BS}^{(a)}_{EQA}(A)$. Then the following hold for all $x, y\in A$: \\
$(1)$ $s(x\thicksim y)=s(y\backsim x)=1+s(x)-s(y)$ and $s(x)\le s(y)$ whenever $x\le y;$ \\
$(2)$ $s(x)=0$ and $s(x\thicksim a)=s(a\backsim x)=1$ for all $x\le a;$ \\
$(3)$ $s(x^{\thicksim_a})=s(x^{\backsim_a})=1-s(x)$ and  
      $s(x^{\thicksim_a\backsim_a})=s(x^{\backsim_a\thicksim_a})=s(x)$ for all $x\ge a$. 
\end{prop}
\begin{proof}
$(1)$ Since $x\le y$, from conditions $(BS_1)$ and $(BS_2)$ we get: \\
$\hspace*{2cm}$ $s(x)+1=s(y)+s(x\thicksim y)$, \\
$\hspace*{2cm}$ $s(x)+1=s(y)+s(y\backsim x)$, \\
hence $s(x\thicksim y)=s(y\backsim x)=1+s(x)-s(y)$. \\
It follows that $s(x)-s(y)=s(x\thicksim y)-1\le 0$, thus $s(x)\le s(y)$. \\
$(2)$ Taking $y=a$ in $(1)$ we have: $s(x\thicksim a)=s(a\backsim x)=1+s(x)-s(a)=1+s(x)$. \\
Since $s(x\thicksim a)=s(a\backsim x)\le 1$, we get $s(x)=0$ and $s(x\thicksim a)=s(a\backsim x)=1$. \\
$(3)$ By $(1)$ we have $s(a\thicksim x)=s(x\backsim a)=1+s(a)-s(x)$, that is $s(x^{\thicksim_a})=s(x^{\backsim_a})=1-s(x)$. 
Applying these identities we get $s(x^{\thicksim_a\backsim_a})=s(x^{\backsim_a\thicksim_a})=s(x)$. 
\end{proof}

\begin{prop} \label{bs-eq-30} Let $(A, \wedge, \thicksim, \backsim, a, 1)$ be a pointed pseudo equality algebra and 
$s:A\longrightarrow [0,1]$ be a function such that $s(a)=0$. Then the following are equivalent: \\
$(a)$ $s\in \mathcal{BS}^{(a)}_{EQA}(A);$ \\
$(b)$ $s(y\thicksim x)=s(x\backsim y)=1-s(x)+s(y)$, for all $x, y\in A$ such that $y\le x$.
\end{prop}
\begin{proof}
$(a)\Rightarrow (b)$ It follows from $(BS_1)$ and $(BS_2)$. \\
$(b)\Rightarrow (a)$ Applying $(b)$, since $x\wedge y\le x$ and $x\wedge y\le y$, we get: \\
$\hspace*{2cm}$ $s(x\wedge y\thicksim x)=s(x\backsim x\wedge y)=1-s(x)+s(x\wedge y)$ \\
$\hspace*{2cm}$ $s(x\wedge y\thicksim y)=s(y\backsim x\wedge y)=1-s(y)+s(x\wedge y)$. \\
It follows that: \\
$\hspace*{1cm}$ $s(x)+s(x\wedge y\thicksim x)=s(x)+1-s(x)+s(x\wedge y)=s(y)+1-s(y)+s(x\wedge y)$ \\
$\hspace*{4.4cm}$ $=s(y)+s(x\wedge y\thicksim y)$ \\
$\hspace*{1cm}$ $s(x)+s(x\backsim x\wedge y)=s(x)+1-s(x)+s(x\wedge y)=s(y)+1-s(y)+s(x\wedge y)$ \\
$\hspace*{4.4cm}$ $=s(y)+s(y\backsim x\wedge y)$, \\
that is, $(BS_1)$ and $(BS_2)$. \\
Moreover, by $(b)$ we have $s(1)=s(x\thicksim x)=1-s(x)+s(x)=1$. \\
We conclude that $s\in \mathcal{BS}^{(a)}_{EQA}(A)$. 
\end{proof}

The Bosbach states on bounded pseudo BCK-algebras have been investigated in \cite{Ciu6}. 
We extend this notion to the case of pointed pseudo BCK-algebras. 

\begin{Def} \label{bs-eq-40}
Let $(B, \wedge, \rightarrow, \rightsquigarrow, a, 1)$ be a pointed pseudo BCK-algebra with $a\ne 1$. 
A \emph{Bosbach state} on $B$ is a function $s:B\longrightarrow [0,1]$ satisfying the following axioms: \\
$(BS_1^{\prime})$ $s(x)+s(x\rightarrow y)=s(y)+s(y\rightarrow x);$ \\
$(BS_2^{\prime})$ $s(x)+s(x\rightsquigarrow y)=s(y)+s(y\rightsquigarrow x);$ \\
$(BS_3^{\prime})$ $s(1)=1$ and $s(a)=0$, \\
for all $x, y\in B$. 
\end{Def}

Denote $\mathcal{BS}^{(a)}_{BCK}(B)$ the set of all Bosbach states on a pointed pseudo BCK-algebra 
$(B, \wedge, \rightarrow, \rightsquigarrow, a, 1)$. 

\begin{prop} \label{bs-eq-50} For any pointed pseudo equality algebra $(A, \wedge, \thicksim, \backsim, a, 1)$, 
$\mathcal{BS}^{(a)}_{EQA}(A)\subseteq \mathcal{BS}^{(a)}_{BCK}(\Psi(A))$. 
\end{prop}
\begin{proof} Let $s\in \mathcal{BS}^{(a)}_{EQA}(A)$. 
Replacing $x\wedge y\thicksim x=x\rightarrow y$, $x\wedge y\thicksim y=y\rightarrow x$ and 
$x\backsim x\wedge y=x\rightsquigarrow y$, $y\backsim x\wedge y=y\rightsquigarrow x$ in $(BS_1)$ and $(BS_2)$, respectively, we get $(BS_1^{\prime})$ and $(BS_2^{\prime})$ for $\Psi(A)$.
Hence $\mathcal{BS}^{(a)}_{EQA}(A)\subseteq \mathcal{BS}^{(a)}_{BCK}(\Psi(A))$.
\end{proof} 

\begin{theo} \label{bs-eq-60} For any pointed invariant pseudo equality algebra 
$(A, \wedge, \thicksim, \backsim, a, 1)$, \\ 
$\mathcal{BS}^{(a)}_{EQA}(A)=\mathcal{BS}^{(a)}_{BCK}(\Psi(A))$. 
\end{theo}
\begin{proof} Acording to Proposition \ref{bs-eq-50}, we have 
$\mathcal{BS}^{(a)}_{EQA}(A)\subseteq \mathcal{BS}^{(a)}_{BCK}(\Psi(A))$ for any pointed pseudo equality algebra 
$(A, \wedge, \thicksim, \backsim, a, 1)$. \\
Let $A$ be a pointed invariant pseudo equality algebra and let $s\in \mathcal{BS}^{(a)}_{BCK}(\Psi(A))$. \\
Taking into consideration that $x\rightarrow y=y\thicksim x$ and $x\rightsquigarrow y=x\backsim y$,  
for all $x, y\in A$, axioms $(BS_1^{\prime})$ and $(BS_2^{\prime})$ become: \\
$\hspace*{3cm}$ $s(x)+s(y\thicksim x)=s(y)+s(x\thicksim y)$ and \\
$\hspace*{3cm}$ $s(x)+s(x\backsim y)=s(y)+s(y\backsim x)$, \\
respectively. \\
Applying Proposition \ref{ps-eq-150}, we get: \\ 
$\hspace*{3cm}$ $s(x)+s(x\wedge y\thicksim x)=s(y)+s(x\wedge y\thicksim y)$ and \\
$\hspace*{3cm}$ $s(x)+s(x\backsim x\wedge y)=s(y)+s(y\backsim x\wedge y)$, \\
for all $x, y\in A$. 
It follows that axioms $(BS_1)$ and $(BS_2)$ are satisfied. \\
Since $(BS_3)$ is the same as $(BS_3^{\prime})$, it follows that $s\in \mathcal{BS}^{(a)}_{EQA}(A)$. \\
We conclude that $\mathcal{BS}^{(a)}_{EQA}(A)=\mathcal{BS}^{(a)}_{BCK}(\Psi(A))$.
\end{proof}

\begin{ex} \label{bs-eq-80} Consider the invariant pointed pseudo equality algebra 
$(B, \wedge, \thicksim, \backsim, \alpha, 1)$ from Example \ref{ps-ex-50}, with $\alpha\in \{0, a, b\}$. 
Define $s_u:B\longrightarrow [0, 1]$, by $s(0)=0$, $s(a)=u$, $s(b)=1-u$, $s(1)=1$, where $u\in [0, 1]$.   
Then $\mathcal{BS}^{(0)}_{EQA}(B)=\mathcal{BS}^{(0)}_{BCK}(\Psi(B))=\{s_u\mid u\in [0, 1]\}$.
Moreover, $\mathcal{BS}^{(a)}_{EQA}(B)=\mathcal{BS}^{(a)}_{BCK}(\Psi(B))=\{s_0\}$ and 
$\mathcal{BS}^{(b)}_{EQA}(B)=\mathcal{BS}^{(b)}_{BCK}(\Psi(B))=\{s_1\}$.
\end{ex}

$\vspace*{5mm}$

\section{States pseudo equality algebras}

In this section we define and study two types of internal states on pseudo equality algebras and their 
corresponding pseudo BCK(pC)-meet-semilattices, and we investigate the connections between the internal states on the two structures. We prove that any internal state on a pseudo equality algebra is also an internal state on its corresponding pseudo BCK(pC)-meet-semilattice, and we prove the converse for the case of linearly ordered symmetric pseudo equality algebras. We also show that any internal state on a pseudo BCK(pC)-meet-semilattice is also an internal state on its corresponding pseudo equality algebra.

\begin{Def} \label{is-eq-10} Let $(A,\wedge,\thicksim,\backsim,1)$ be a pseudo equality algebra and 
$\sigma:A \longrightarrow A$ be a unary operator on $A$. For all $x, y\in A$ consider the following axioms: \\ 
$(IS_1)$ $\sigma (x)\le \sigma (y)$, whenever $x\le y;$ \\
$(IS_2)$ $\sigma(x\wedge y\thicksim x)=\sigma(y)\thicksim \sigma((x\wedge y\thicksim x) \backsim y)$ and 
         $\sigma(x\backsim x\wedge y)=\sigma(y\thicksim (x\backsim x\wedge y)) \backsim \sigma(y);$ \\
$(IS^{'}_2)$ $\sigma(x\wedge y\thicksim x)=\sigma(y)\thicksim \sigma((x\wedge y\thicksim y) \backsim x)$ and 
             $\sigma(x\backsim x\wedge y)=\sigma(x\thicksim (y\backsim x\wedge y)) \backsim \sigma(y);$ \\
$(IS_3)$ $\sigma(\sigma (x) \thicksim \sigma (y))= \sigma (x) \thicksim \sigma (y)$ and 
         $\sigma(\sigma (x) \backsim \sigma (y))= \sigma (x) \backsim \sigma (y);$ \\
$(IS_4)$ $\sigma(\sigma (x) \wedge \sigma (y))= \sigma (x) \wedge \sigma (y)$. \\
Then: \\
$(i)$ $\sigma$ is called an \emph{internal state of type I} or a \emph{state operator of type I} or 
a \emph{type I state operator} if it satisfies axioms $(IS_1)$, $(IS_2)$, $(IS_3), (IS_4);$ \\    
$(ii)$ $\sigma$ is called an \emph{internal state of type II} or a \emph{state operator of type II} or a 
\emph{type II state operator} if it satisfies axioms $(IS_1)$, $(IS^{'}_2)$, $(IS_3), (IS_4)$. \\
The structure $(A,\wedge,\thicksim,\backsim,\sigma,1)$ ($(A,\sigma)$, for short) is called a 
\emph{state pseudo equality algebra of type I (type II)}, respectively.
\end{Def}

Denote $\mathcal{IS}_{EQA}^{(I)}(A)$ and $\mathcal{IS}_{EQA}^{(II)}(A)$ the set of all internal states of 
type I and II on a pseudo equality algebra $A$, respectively. \\
For $\sigma \in \mathcal{IS}_{EQA}^{(I)}(A)$ or $\sigma \in \mathcal{IS}_{EQA}^{(II)}(A)$, 
$\Ker(\sigma)=\{x\in A \mid \sigma(x)=1\}$  is called the \emph{kernel} of $\sigma$. \\
An internal state $\sigma$ on $A$ is said to be \emph{strong} if it satisfies the condition: \\
$(IS_5)$ $\sigma(x\thicksim y)=\sigma (x\backsim y)$ for all $x, y \in A$. 

\begin{ex} \label{is-eq-20} Let $(A,\wedge,\thicksim,\backsim,1)$ be a pseudo equality algebra and 
${\bf 1}_A, \Id_A:A\longrightarrow A$, defined by ${\bf 1}_A(x)=1$ and $\Id_A(x)=x$ for all $x\in A$. Then: \\
$(1)$ ${\bf 1}_A\in \mathcal{IS}_{EQA}^{(I)}(A), \mathcal{IS}_{EQA}^{(II)}(A);$ \\
$(2)$ $\Id_A \in \mathcal{IS}_{EQA}^{(I)}(A)$ (it is obvious that $(IS_1)$, $(IS_3)$ and $(IS_4)$ are verified, 
while $(IS_2)$ follows by Proposition \ref{ps-eq-30}).
\end{ex}

\begin{ex}\label{is-eq-30} Let $(A_1, \wedge_1, \thicksim_1, \backsim_1, 1_1)$ and 
$(A_2, \wedge_2, \thicksim_2, \backsim_2, 1_2)$ be two pseudo equality algebras and $A$ be the pseudo equality 
algebra defined in Example \ref{ps-ex-70}. 
Let $\sigma_1:A_1\longrightarrow A_1$ and $\sigma_2:A_2\longrightarrow A_2$ be internal states of type I (type II) 
on $A_1$ and $A_2$, respectively. Then the map $\sigma:A\longrightarrow A$ defined by $\sigma(x,y)=(\sigma_1(x),\sigma_2(y))$, for all $(x,y)\in A$ is an intenal state of type I (type II) on $A$.
\end{ex}

\begin{rems} \label{is-eq-40} 
$(1)$ If $(A,\wedge,\thicksim,\backsim,1)$ is a commutative pseudo equality algebra, 
then $\mathcal{IS}_{EQA}^{(I)}(A)=\mathcal{IS}_{EQA}^{(II)}(A)$. \\
$(2)$ In general, $\mathcal{IS}_{EQA}^{(I)}(A)\ne \mathcal{IS}_{EQA}^{(II)}(A)$. \\
Indeed, let $A=\{0,a,b,1\}$ with $0<a<b<1$ and consider the operation $\thicksim$ given by the following table: 
\[
\hspace{10mm}
\begin{array}{c|ccccc}
\thicksim& 0 & a & b & 1 \\ \hline
0 & 1 & a & 0 & 0 \\ 
a & a & 1 & a & a \\ 
b & 0 & a & 1 & b \\  
1 & 0 & a & b & 1
\end{array}
.
\]

One can easily check that $\mathcal{A}=(A, \wedge, \thicksim, 1)$ is a linearly ordered equality algebra 
(see \cite{Ciu5}). \\ 
Since $a\wedge b\thicksim a=1\ne b=b\thicksim ((a\wedge b\backsim b)\thicksim a)$, it follows that 
$\Id_A\notin \mathcal{IS}_{EQA}^{(II)}(A)$, that is $\mathcal{IS}_{EQA}^{(I)}(A)\ne \mathcal{IS}_{EQA}^{(II)}(A)$.
\end{rems}

\begin{ex} \label{is-eq-50} Consider the commutative pseudo equality algebra $(B, \wedge, \thicksim, \backsim, 1)$ 
from Example \ref{ps-ex-50} and the maps $\sigma_i:B\longrightarrow B$, $i=1,\cdots,6$ given in the table below:
\[
\begin{array}{c|ccccc}
 x & 0 & a & b & 1 \\ \hline
\sigma_1(x) & 0 & 0 & 1 & 1 \\
\sigma_2(x) & 0 & a & b & 1 \\ 
\sigma_3(x) & 0 & 1 & 0 & 1 \\ 
\sigma_4(x) & a & a & 1 & 1 \\
\sigma_5(x) & b & 1 & b & 1 \\
\sigma_6(x) & 1 & 1 & 1 & 1  
\end{array}
. 
\]
Then $\mathcal{IS}_{EQA}^{(I)}(B)=\mathcal{IS}_{EQA}^{(II)}(B)=
\{\sigma_1, \sigma_2, \sigma_3, \sigma_4, \sigma_5, \sigma_6\}$. 
This is in accordance with Remark \ref{is-eq-40}$(1)$.
\end{ex}

\begin{prop} \label{is-eq-60} If $(A,\sigma)$ is a state pseudo equality algebra of type I or type II, 
then for all $x, y\in A$ the following hold:\\
$(1)$ $\sigma(1)=1;$ \\
$(2)$ $\sigma(\sigma(x))=\sigma(x);$ \\
$(3)$ $\sigma(A)=\{x\in A \mid x=\sigma(x)\};$ \\ 
$(4)$ $\sigma(A)$ is a subalgebra of $A$. 
\end{prop}
\begin{proof}
$(1)$ It follows from $(IS_3)$ for $y=x;$ \\
$(2)$ Applying $(IS_3)$ we get:\\ 
$\hspace*{1cm}$
$\sigma(\sigma(x))=\sigma(\sigma(x)\thicksim 1)=\sigma(\sigma(x)\thicksim \sigma(1))=
\sigma(x)\thicksim \sigma(1)=\sigma (x)\thicksim 1=\sigma (x)$. \\
$(3)$ Clearly $\{x\in A \mid x=\sigma(x)\} \subseteq \sigma(A)$. 
Let $x\in \sigma(A)$, that is there exists $x_1\in A$ such that $x=\sigma(x_1)$. 
It follows that $x=\sigma(x_1)=\sigma(\sigma(x_1))=\sigma(x)$, that is $x\in \sigma(A)$. \\
Thus $\sigma(A) \subseteq \{x\in A \mid x=\sigma(x)\}$ and we  conclude that $\sigma(A)=\{x\in A \mid x=\sigma(x)\}$. \\
$(4)$ By $(1)$, $1\in \sigma(A)$. Let $x, y\in A$. Then by $(IS_3)$ and $(IS_4)$ it follows that 
$\sigma(x)\thicksim \sigma(y), \sigma(x)\backsim \sigma(y),  \sigma(x)\wedge \sigma(y)\in \sigma(A)$, that is 
$\sigma(A)$ is a subalgebra of $A$. 
\end{proof}

\begin{prop} \label{is-eq-70} Let $(A,\sigma)$ be a state pseudo equality algebra of type I or type II 
and $x, y\in A$ such that $y\le x$. Then the following hold: \\
$(1)$ $\sigma(y\thicksim x)\le \sigma(y)\thicksim \sigma (x)$ and 
      $\sigma(x\backsim  y)\le \sigma(x)\backsim  \sigma (y);$ \\
$(2)$ $\sigma(y)\thicksim \sigma (x) = \sigma(x)\rightarrow \sigma (y)$ and 
      $\sigma(x)\backsim \sigma (y) = \sigma(x)\rightsquigarrow\sigma (y);$ \\
$(3)$ $\sigma(x \rightarrow y) \le \sigma(x)\rightarrow \sigma (y)$ and 
      $\sigma(x \rightsquigarrow y) \le  \sigma(x)\rightsquigarrow \sigma (y)$. 
\end{prop}
\begin{proof} 
$(1)$ We will apply Proposition \ref{ps-eq-20} and consider two cases. \\ 
$(I)$ Suppose that $\sigma\in \mathcal{IS}_{EQA}^{(I)}(A)$. \\ 
From $y\le x\le (x\wedge y\thicksim x)\backsim y$ it follows that 
$\sigma(y)\le \sigma(x)\le \sigma((x\wedge y\thicksim x)\backsim y)$. \\
Applying $(A_4)$ and $(IS_2)$ we get 
$\sigma(y)\thicksim \sigma((x\wedge y\thicksim x)\backsim y)\le \sigma(y)\thicksim \sigma (x)$, 
that is $\sigma(x\wedge y\thicksim x)\le \sigma(y)\thicksim \sigma (x)$, so that  
$\sigma(y\thicksim x)\le \sigma(y)\thicksim \sigma (x)$. \\
From $y\le x\le y\thicksim (x\backsim x\wedge y)$ it follows that 
$\sigma(y)\le \sigma(x)\le \sigma(y\thicksim (x\backsim x\wedge y))$. \\
Applying $(A_4)$ and $(IS_2)$ we get 
$\sigma(y\thicksim (x\backsim x\wedge y))\backsim \sigma(y)\le \sigma(x)\backsim \sigma (y)$, that is \\ 
$\sigma(x\backsim x\wedge y)\le \sigma(x)\backsim \sigma (y)$, so that  
$\sigma(x\backsim y)\le \sigma(x)\backsim \sigma (y)$. \\
$(II)$ Suppose that $\sigma\in \mathcal{IS}_{EQA}^{(II)}(A)$. \\
From $y\le x\le (x\wedge y\thicksim y)\backsim x$ it follows that 
$\sigma(y)\le \sigma(x)\le \sigma((x\wedge y\thicksim y)\backsim x)$. \\
Applying $(A_4)$ and $(IS^{'}_2)$ we get 
$\sigma(y)\thicksim \sigma((x\wedge y\thicksim y)\backsim x)\le \sigma(y)\thicksim \sigma (x)$, 
that is $\sigma(x\wedge y\thicksim x)\le \sigma(y)\thicksim \sigma (x)$, so that  
$\sigma(y\thicksim x)\le \sigma(y)\thicksim \sigma (x)$. \\
From $y\le x\le x\thicksim (y\backsim x\wedge y)$ it follows that 
$\sigma(y)\le \sigma(x)\le \sigma(x\thicksim (y\backsim x\wedge y))$. \\
Applying $(A_4)$ and $(IS^{'}_2)$ we get 
$\sigma(x\thicksim (y\backsim x\wedge y))\backsim \sigma(y)\le \sigma(x)\backsim \sigma (y)$, that is \\ 
$\sigma(x\backsim x\wedge y)\le \sigma(x)\backsim \sigma (y)$, so that  
$\sigma(x\backsim y)\le \sigma(x)\backsim \sigma (y)$. \\
$(2)$ Since $y\le x$, we have $\sigma(y)\le \sigma(x)$ and 
$\sigma(x)\rightarrow \sigma (y)=\sigma(x)\wedge \sigma (y)\thicksim \sigma(x)=\sigma(y)\thicksim \sigma (x)$. \\ 
Similarly $\sigma(x)\rightsquigarrow\sigma (y)=\sigma(x)\backsim \sigma(x)\wedge \sigma (y)= 
\sigma(x)\backsim \sigma (y)$. \\
$(3)$ Applying $(1)$ and $(2)$ we get: \\
$\sigma(x \rightarrow y) = \sigma(x\wedge y \thicksim x)=\sigma(y\thicksim x)\le 
\sigma(y)\thicksim \sigma (x) =\sigma (x)\wedge \sigma(y)\thicksim \sigma (x) \le \sigma(x)\rightarrow \sigma (y)$. \\ 
$\sigma(x \rightsquigarrow y)=\sigma(x\backsim x\wedge y)=\sigma(x\backsim y)\le 
\sigma(x)\backsim \sigma(y)=\sigma(x)\backsim \sigma(x)\wedge \sigma(y)=\sigma(x)\rightsquigarrow \sigma (y)$.
\end{proof}

\begin{prop} \label{is-eq-80} If $(A,\sigma)$ is a state pseudo equality algebra of type I or type II, then 
the following hold for all $x, y\in A$:\\
$(1)$ $\Ker(\sigma)\cap \Img(\sigma)=\{1\};$ \\
$(2)$ $\Ker(\sigma)\in {\mathcal DS}(A);$ \\
$(3)$ if $A$ is invariant, then $\Ker(\sigma)$ is a subalgebra of $A;$ \\
$(4)$ if $\sigma$ is strong, then $\Ker(\sigma)\in {\mathcal DS}_n(A)$. 
\end{prop}
\begin{proof}
$(1)$ Let $x\in \Ker(\sigma)\cap \Img(\sigma)$. It follows that $\sigma(x)=1$ and there exists $x_1\in A$ 
such that $x=\sigma(x_1)$. Hence $1=\sigma(x)=\sigma(\sigma(x_1))=\sigma(x_1)=x$. Thus $x=1$ and we 
conclude that $\Ker(\sigma)\cap \Img(\sigma)=\{1\}$. \\
$(2)$ Obviously $1\in \Ker(\sigma)$. \\
Consider $x, y\in A$ such that $x\in \Ker(\sigma)$ and $x\le y$. \\
Then by $(IS_1)$, $1=\sigma(x)\le \sigma(y)$, so $\sigma(y)=1$. Hence $y\in \Ker(\sigma)$. \\
Take $x, y\in A$ such that $x, y\thicksim x \in \Ker(\sigma)$, that is 
$\sigma(x)=\sigma(y\thicksim x)=1$. \\
Since by Proposition \ref{ps-eq-20}$(5)$ we have 
$y\thicksim x \le x\wedge y \thicksim x$, it follows that $1=\sigma(y\thicksim x) \le \sigma(x\wedge y \thicksim x)$, 
thus $\sigma(x\wedge y \thicksim x)=1$. 
By Proposition \ref{ps-eq-20}$(1)$,$(3)$ we have $y\le x\wedge y\thicksim x$ and 
$x\le (x\wedge y\thicksim x)\backsim y$. 
Applying Proposition \ref{is-eq-70} we get   
$1=\sigma(x)\le \sigma((x\wedge y\thicksim x)\backsim y)\le\sigma(x\wedge y\thicksim x)\backsim \sigma(y)=
1\backsim \sigma(y)=\sigma(y)$. It follows that $\sigma(y)=1$, that is $y\in \Ker(\sigma)$. \\ 
We conclude that $\Ker(\sigma)$ is a deductive system of $A$. \\
$(3)$ It is a corollary of $(2)$ and Proposition \ref{lds-eqa-45}. \\ 
$(4)$ If $x\thicksim y \in \Ker(\sigma)$, then $\sigma(x\backsim y)=\sigma(x\thicksim y)=1$, hence 
$x\backsim y \in \Ker(\sigma)$. Similarly from $x\backsim y \in \Ker(\sigma)$ we get $x\thicksim y \in \Ker(\sigma)$, 
thus $\Ker(\sigma)\in {\mathcal DS}_n(A)$. 
\end{proof}

Let $(A, \wedge, \thicksim, \backsim, 1)$ be a pseudo equality algebra and $\sigma$ be a strong internal 
state of type I or type II on $A$. 
Denote $K=\Ker(\sigma)$. 
Since $K\in {\mathcal DS}_n(A)$, it follows that $\Theta_K\in {\mathcal Con}(A)$. 
According to \cite{Dvu7}, $(A/{\Theta_K},\thicksim,\backsim,1/\Theta_K)$ is a pseudo equality algebra with the natural operations induced from those of $A$. \\ 
In what follows we define the notion of an internal state on pseudo BCK-meet-semilattices and we investigate 
the connection between the internal states on a pseudo equality algebra $(A, \wedge, \thicksim, \backsim, 1)$ 
and the internal states on its corresponding pseudo BCK(pC)-meet-semilattice 
$\Psi(A)=(A, \wedge, \rightarrow, \rightsquigarrow, 1)$. 
For more details regarding the internal states on pseudo BCK-algebras we refer the reader to \cite{Ciu7}.

\begin{Def} \label{is-eq-110} $\rm($\cite{Ciu7}$\rm)$ Let $(B, \wedge, \rightarrow, \rightsquigarrow, 1)$ be a pseudo BCK-meet-semilattice and $\mu:B \longrightarrow B$ be a unary operator on $B$. For all $x, y\in B$ consider the following axioms:\\
$(SB_1)$ $\mu(x)\le \mu(y)$, whenever $x\le y,$ \\
$(SB_2)$ $\mu(x\rightarrow y)=\mu((x\rightarrow y)\rightsquigarrow y)\rightarrow \mu(y)$ and 
         $\mu(x\rightsquigarrow y)=\mu((x\rightsquigarrow y)\rightarrow y)\rightsquigarrow \mu(y),$ \\
$(SB^{'}_2)$ $\mu(x\rightarrow y)=\mu((y\rightarrow x)\rightsquigarrow x)\rightarrow \mu(y)$ and
             $\mu(x\rightsquigarrow y)=\mu((y\rightsquigarrow x)\rightarrow x)\rightsquigarrow \mu(y),$ \\ 
$(SB_3)$ $\mu(\mu(x)\rightarrow \mu(y))=\mu(x)\rightarrow \mu(y)$ and 
         $\mu(\mu(x)\rightsquigarrow \mu(y))=\mu(x)\rightsquigarrow \mu(y),$ \\ 
$(SB_4)$ $\mu(\mu(x)\wedge \mu(y))=\mu(x)\wedge \mu(y)$. \\
Then: \\
$(i)$ $\mu$ is called an \emph{internal state of type I} or a \emph{state operator of type I} or 
a \emph{type I state operator} if it satisfies axioms $(SB_1)$, $(SB_2)$, $(SB_3)$, $(SB_4);$ \\    
$(ii)$ $\mu$ is called an \emph{internal state of type II} or a \emph{state operator of type II} or a 
\emph{type II state operator} if it satisfies axioms $(SB_1)$, $(SB^{'}_2)$, $(SB_3)$, $(SB_4)$. \\
The structure $(B, \rightarrow, \rightsquigarrow, \mu, 1)$ ($(B,\mu)$, for short) is called a 
\emph{state pseudo BCK-meet-semilattice of type I (type II)}, respectively.
\end{Def}

Denote $\mathcal{IS_{BCK}}^{(I)}(B)$ and $\mathcal{IS}_{BCK}^{(II)}(B)$ the set of all internal states of 
type I and II on a pseudo BCK-meet-semilattice $B$, respectively. \\
For $\mu \in \mathcal{IS}_{BCK}^{(I)}(B)$ or $\mu \in \mathcal{IS}_{BCK}^{(II)}(B)$, 
$\Ker(\mu)=\{x\in B \mid \mu(x)=1\}$  is called the \emph{kernel} of $\mu$. 

\begin{theo} \label{is-eq-120} Let $(A, \wedge, \thicksim, \backsim, 1)$ be a pseudo equality algebra and let $\Psi(A)=(A,\wedge,\rightarrow,\rightsquigarrow,1)$ be its corresponding pseudo BCK(pC)-meet-semilattice. 
Then $\mathcal{IS}_{EQA}^{(I)}(A)\subseteq \mathcal{IS}_{BCK}^{(I)}(\Psi(A))$ and 
$\mathcal{IS}_{EQA}^{II}(A)\subseteq \mathcal{IS}_{BCK}^{II}(\Psi(A))$.
\end{theo}
\begin{proof} Consider $\sigma:A \longrightarrow A$ satisfying $(IS_1)$, $(IS_2)$, $(IS^{'}_2)$, $(IS_3)$, $(IS_4)$. \\ 
Obviously $(SB_1)$ is satisfied due to $(IS_1)$. \\ 
By Proposition \ref{ps-eq-20}$(1)$, $y\le x\wedge y\thicksim x$, so $(x\wedge y\thicksim x)\wedge y=y$. \\ 
From Proposition \ref{ps-eq-20}$(4)$ we have $y\le (x\wedge y\thicksim x)\backsim y$, thus 
$\sigma(y)\le \sigma(((x\wedge y\thicksim x)\backsim y)$. \\
Applying the definition of $x\rightarrow y$ and $(IS_2)$ we get:\\ 
$\hspace*{0.5cm}$
$\sigma (x\rightarrow y)=\sigma(x\wedge y\thicksim x)=
\sigma(y)\thicksim \sigma((x\wedge y\thicksim x) \backsim y)$ \\
$\hspace*{2.2cm}$
$=\sigma((x\wedge y\thicksim x)\thicksim y)\wedge \sigma(y)\thicksim \sigma((x\wedge y\thicksim x) \backsim y)$ \\  
$\hspace*{2.2cm}$
$=\sigma((x\wedge y\thicksim x)\thicksim (x\wedge y\thicksim x)\wedge y)\wedge \sigma(y)
\thicksim \sigma((x\wedge y\thicksim x) \backsim (x\wedge y\thicksim x)\wedge y)$ \\
$\hspace*{2.2cm}$
$=\sigma((x\rightarrow y)\rightsquigarrow y)\rightarrow \sigma(y)$. \\
Similarly $\sigma (x\rightsquigarrow y)=\sigma((x\rightsquigarrow y)\rightarrow y)\rightsquigarrow \sigma(y)$, 
thus $(SB_2)$ is satisfied. \\
By Proposition \ref{ps-eq-20}$(1)$, $x\le x\wedge y\thicksim y$, so $(x\wedge y\thicksim y)\wedge x=x$. \\ 
From Proposition \ref{ps-eq-20}$(3)$ we have $y\le (x\wedge y\thicksim y)\backsim x$, thus 
$\sigma(y)\le \sigma(((x\wedge y\thicksim y)\backsim x)$. \\
Applying $(IS^{'}_2)$ we get:\\ 
$\hspace*{0.5cm}$
$\sigma (x\rightarrow y)=\sigma(x\wedge y\thicksim x)=
\sigma(y)\thicksim \sigma((x\wedge y\thicksim y) \backsim x)$ \\
$\hspace*{2.2cm}$
$=\sigma((x\wedge y\thicksim y)\thicksim x)\wedge \sigma(y)\thicksim \sigma((x\wedge y\thicksim y) \backsim x)$ \\  
$\hspace*{2.2cm}$
$=\sigma((x\wedge y\thicksim y)\thicksim (x\wedge y\thicksim y)\wedge x)\wedge \sigma(y)
\thicksim \sigma((x\wedge y\thicksim y) \backsim (x\wedge y\thicksim y)\wedge x)$ \\
$\hspace*{2.2cm}$
$=\sigma((y\rightarrow x)\rightsquigarrow x)\rightarrow \sigma(y)$. \\
Similarly $\sigma (x\rightsquigarrow y)=\sigma((y\rightsquigarrow x)\rightarrow x)\rightsquigarrow \sigma(y)$, 
thus $(SB^{'}_2)$ is satisfied. \\
From $(IS_3)$ and $(IS_4)$ we have:\\
$\hspace*{0.5cm}$
$\sigma (\sigma(x)\rightarrow \sigma(y))=\sigma (\sigma(x)\wedge \sigma(y)\thicksim \sigma(x)) =
\sigma(\sigma(x)\wedge \sigma(y))\thicksim \sigma((\sigma(x))$ \\
$\hspace*{3.2cm}$ 
$=\sigma(x)\wedge \sigma(y)\thicksim \sigma(x)=\sigma(x)\rightarrow \sigma(y)$. \\
Similarly $\sigma (\sigma(x)\rightsquigarrow \sigma(y))=\sigma(x)\rightsquigarrow \sigma(y)$, hence $(SB_3)$. \\
Since $(SB_4)$ is in fact $(IS_4)$, it follows that $\sigma$ satisfies $(SB_1)$, $(SB_2)$, $(SB^{'}_2)$, $(SB_3)$, $(SB_4)$. 
We conclude that $\mathcal{IS}_{EQA}^{(I)}(A)\subseteq \mathcal{IS}_{BCK}^{(I)}(\Psi(A))$ and 
$\mathcal{IS}_{EQA}^{(II)}(A)\subseteq \mathcal{IS}_{BCK}^{(II)}(\Psi(A))$.
\end{proof}

\begin{theo} \label{is-eq-130} Let $(A, \wedge, \thicksim, \backsim, 1)$ be a linearly ordered symmetric pseudo  equality algebra and let $\Psi(A)=(A,\wedge,\rightarrow,\rightsquigarrow,1)$ be its corresponding pseudo 
BCK(pC)-meet-semilattice. Then $\mathcal{IS}_{BCK}^{(I)}(\Psi(A))\subseteq \mathcal{IS}_{EQA}^{(I)}(A)$ and 
$\mathcal{IS}_{BCK}^{(II)}(\Psi(A))\subseteq \mathcal{IS}_{EQA}^{(II)}(A)$. 
\end{theo}
\begin{proof} Consider $\mu:\Psi(A) \longrightarrow \Psi(A)$ satisfying $(SB_1)$, $(SB_2)$, $(SB^{'}_2)$, $(SB_3)$, $(SB_4)$ and let $x, y\in A$. 
Obviously $(SB_1)$ and $(SB_4)$ are satisfied due to $(IS_1)$ and $(IS_4)$, respectively. \\ 
Applying axioms $(SB_2)$, $(B_2)$ and Lemma \ref{ps-eq-40}$(4)$, we have: \\
$\hspace*{1cm}$
$\mu(x\wedge y\thicksim x)=\mu(x\rightarrow y)=\mu((x\rightarrow y)\rightsquigarrow y)\rightarrow \mu(y) $\\
$\hspace*{3.2cm}$
$=\mu((x\rightarrow y)\rightsquigarrow y)\wedge \mu(y)\thicksim \mu((x\rightarrow y)\rightsquigarrow y)=
\mu(y)\thicksim \mu((x\rightarrow y)\rightsquigarrow y)$ \\
$\hspace*{3.2cm}$
$=\mu(y)\thicksim \mu((x\rightarrow y)\backsim (x\rightarrow y)\wedge y)= 
\mu(y)\thicksim \mu((x\wedge y\thicksim x)\backsim y)$. \\
$\hspace*{1cm}$
$\mu(x\backsim x\wedge y)=\mu(x\rightsquigarrow y)=\mu((x\rightsquigarrow y)\rightarrow y)\rightsquigarrow \mu(y) $\\
$\hspace*{3.2cm}$
$=\mu((x\rightsquigarrow y)\rightarrow y)\backsim \mu((x\rightsquigarrow y)\rightarrow y)\wedge \mu(y)=
\mu((x\rightsquigarrow y)\rightarrow y)\backsim \mu(y)$ \\ 
$\hspace*{3.2cm}$
$=\mu((x\rightsquigarrow y)\wedge y\thicksim (x\rightsquigarrow y))\backsim \mu(y)=
\mu(y\thicksim (x\backsim x\wedge y))\backsim \mu(y)$. \\
Hence $\mu$ satisfies $(IS_2)$. \\ 
Similarly, by $(SB^{'}_2)$ we get: \\
$\mu(x\wedge y\thicksim x)=\mu(y)\thicksim \mu((x\wedge y\thicksim y) \backsim x)$ and 
$\mu(x\backsim x\wedge y)=\mu(x\thicksim (y\backsim x\wedge y)) \backsim \mu(y)$, \\ 
thus $\mu$ satisfies $(IS^{'}_2)$. \\
For axiom $(IS_3)$ we will apply the fact that $A$ is a symmetric pseudo equality algebra and consider two cases: \\
$(I)$ Suppose $x\le y$, hence $\mu(x)\le \mu(y)$. We have: \\ 
$\hspace*{1cm}$
$\mu(\mu(x)\thicksim \mu(y))=\mu(\mu(x)\wedge \mu(y)\thicksim \mu(y))=\mu(\mu(y)\rightarrow \mu(x))= 
\mu(y)\rightarrow \mu(x)$ \\
$\hspace*{3.6cm}$
$=\mu(x)\wedge \mu(y)\thicksim \mu(y)=\mu(x)\thicksim \mu(y)$. \\ 
$\hspace*{1cm}$
$\mu(\mu(x)\backsim \mu(y))=\mu(\mu(y)\thicksim \mu(x))=\mu(y)\thicksim \mu(x)=\mu(x)\backsim \mu(y)$. \\
$(II)$ Suppose $y\le x$, so $\mu(y)\le \mu(x)$. Then: \\ 
$\hspace*{1cm}$
$\mu(\mu(x)\backsim \mu(y))=\mu(\mu(x)\backsim \mu(x)\wedge \mu(y))=\mu(\mu(x)\rightsquigarrow \mu(y))= 
\mu(x)\rightsquigarrow \mu(y)$ \\
$\hspace*{3.6cm}$
$=\mu(x)\backsim \mu(x)\wedge \mu(y)=\mu(x)\backsim \mu(y)$. \\
$\hspace*{1cm}$
$\mu(\mu(x)\thicksim \mu(y))=\mu(\mu(y)\backsim \mu(x))=\mu(y)\backsim \mu(x)=\mu(x)\thicksim \mu(y)$. \\
Hence $(IS_3)$ is verified.  
We conclude that $\mathcal{IS}_{BCK}^{(I)}(\Psi(A))\subseteq \mathcal{IS}_{EQA}^{(I)}(A)$ and 
$\mathcal{IS}_{BCK}^{(II)}(\Psi(A))\subseteq \mathcal{IS}_{EQA}^{(I)I}(A)$. 
\end{proof}

\begin{cor} \label{is-eq-140} If $\mathcal{A}=(A, \wedge, \thicksim, \backsim, 1)$ is a linearly ordered 
symmetric pseudo equality algebra, then $\mathcal{IS}_{EQA}^{(I)}(A)=\mathcal{IS}_{BCK}^{(I)}(\Psi(A))$ and 
$\mathcal{IS}_{EQA}^{(II)}(A)=\mathcal{IS}_{BCK}^{(II)}(\Psi(A))$.
\end{cor}

\begin{ex} \label{is-eq-150} Consider the commutative pseudo equality algebra $(B, \wedge, \thicksim, \backsim, 1)$ 
from Example \ref{ps-ex-50} and the maps $\mu_i:B\longrightarrow B$, $i=1,\cdots,6$ given in the table below:
\[
\begin{array}{c|ccccc}
 x & 0 & a & b & 1 \\ \hline
\mu_1(x) & 0 & 0 & 1 & 1 \\
\mu_2(x) & 0 & a & b & 1 \\ 
\mu_3(x) & 0 & 1 & 0 & 1 \\ 
\mu_4(x) & a & a & 1 & 1 \\
\mu_5(x) & b & 1 & b & 1 \\
\mu_6(x) & 1 & 1 & 1 & 1  
\end{array}
. 
\]
Then $\mathcal{IS}_{BCK}^{(I)}(\Psi(B))=\mathcal{IS}_{BCK}^{(II)}(\Psi(B))=\mathcal{IS}_{EQA}^{(I)}(B)=
\mathcal{IS}_{EQA}^{(II)}(B)=\{\mu_1, \mu_2, \mu_3, \mu_4, \mu_5, \mu_6\}$. 
\end{ex}

\begin{theo} \label{is-eq-160} Let $(B, \wedge, \rightarrow, \rightsquigarrow, 1)$ be a pseudo 
BCK(pC)-meet-semilattice and let $\Phi(B)=(B, \wedge, \thicksim=\leftarrow, \backsim=\rightsquigarrow, 1)$ 
be its corresponding pseudo equality algebra. \\ 
Then $\mathcal{IS}^{(I)}_{BCK}(B)\subseteq \mathcal{IS}^{(I)}_{EQA}(\Phi(B))$ and 
$\mathcal{IS}^{II}_{BCK}(B)\subseteq \mathcal{IS}^{II}_{EQA}(\Phi(B))$.
\end{theo}
\begin{proof} Consider $\mu:B\longrightarrow B$ satisfying $(SB_1)$, $(SB_2)$, $(SB^{'}_2)$, $(SB_3)$, $(SB_4)$. \\ 
Axioms $(IS_1)$ and $(IS_4)$ are straightforward. Let $x, y\in A$. \\ 
Applying Remark \ref{ps-eq-100-10} and $(SB_2)$ we have: \\ 
$\hspace*{1cm}$
$\mu(x\wedge y\thicksim x)=\mu(x\rightarrow x\wedge y)=\mu(x\rightarrow y)= 
\mu((x\rightarrow y)\rightsquigarrow y)\rightarrow \mu(y)$ \\
$\hspace*{3.2cm}$
$=\mu(y)\thicksim \mu((x\rightarrow x\wedge y)\rightsquigarrow y)=
\mu(y)\thicksim \mu((x\wedge y\thicksim x)\backsim y)$. \\
$\hspace*{1cm}$
$\mu(x\backsim x\wedge y)=\mu(x\rightsquigarrow x\wedge y)=\mu(x\rightsquigarrow y)= 
\mu((x\rightsquigarrow y)\rightarrow y)\rightsquigarrow \mu(y)$ \\
$\hspace*{3.2cm}$
$=\mu(y\thicksim (x\rightsquigarrow x\wedge y)\backsim \mu(y)=
\mu(y\thicksim (x\backsim x\wedge y)\backsim \mu(y))$. \\
Hence $\mu$ satisfies axiom $(IS_2)$. Similarly \\ 
$\hspace*{1cm}$ 
$\mu(x\wedge y\thicksim x)=\mu(y)\thicksim \mu((x\wedge y\thicksim y) \backsim x)$ and \\
$\hspace*{1cm}$ 
$\mu(x\backsim x\wedge y)=\mu(x\thicksim (y\backsim x\wedge y)) \backsim \mu(y)$, \\ 
thus axiom $(IS^{'}_2)$ is also verified. \\ 
For axiom $(IS_3)$ we have: \\
$\hspace*{1cm}$ 
$\mu(\mu(x)\thicksim \mu(y))=\mu(\mu(y)\rightarrow \mu(x))=\mu(y)\rightarrow \mu(x)=\mu(x)\thicksim \mu(y)$ and \\  
$\hspace*{1cm}$ 
$\mu(\mu(x)\backsim \mu(y))=\mu(\mu(x)\rightsquigarrow \mu(y))=\mu(x)\rightsquigarrow \mu(y)=\mu(x)\backsim \mu(y)$. \\
Hence $\mu$ satisfies axioms $(IS_1)$, $(IS_2)$, $(IS^{'}_2)$, $(IS_3)$, $(IS_4)$. \\
We conclude that  
$\mathcal{IS}^{(I)}_{BCK}(B)\subseteq \mathcal{IS}^{(I)}_{EQA}(\Phi(B))$ and 
$\mathcal{IS}^{II}_{BCK}(B)\subseteq \mathcal{IS}^{II}_{EQA}(\Phi(B))$.
\end{proof}

\begin{ex} \label{is-eq-170} Consider the BCK(C)-lattice $(B,\wedge,\rightarrow,1)$ and its corresponding pseudo
equality algebra $\Phi(B)=(B, \wedge, \thicksim, \backsim, 1)$ from Example \ref{ps-ex-50}.  
Let $\mu_i:B\longrightarrow B$, $i=1,\cdots,6$ be the maps defined in Example \ref{is-eq-150}. 
Then $\mathcal{IS}_{BCK}^{(I)}(B)=\mathcal{IS}_{BCK}^{(II)}(B)=\mathcal{IS}_{EQA}^{(I)}(\Phi(B))=
\mathcal{IS}_{EQA}^{(II)}(\Phi(B))=\{\mu_1, \mu_2, \mu_3, \mu_4, \mu_5, \mu_6\}$. 
\end{ex}

$\vspace*{5mm}$

\section{States-morphism pseudo equality algebras}

In this section we define and study the state-morphism operators on pseudo equality algebras and on their 
corresponding pseudo BCK(pC)-meet-semilattices, and we investigate the connections between the state-morphism operators on the two structures.  
We prove that any state-morphism on a pseudo equality algebra is an internal state of type I. 
It is showen that any state-morphism on a pseudo equality algebra is a state-morphism on its corresponding pseudo BCK(pC)-meet-semilattice, while the converse is true for the case of linearly ordered symmetric pseudo equality 
algebras.
We show that a state-morphism on a pseudo BCK(pC)-meet-semilattice is also a state-morphism on its corresponding pseudo equality algebra.
We also prove that a state-morphism on the set of regular elements on a compatible pseudo equality algebra $A$ can be extended to a state morphism on $A$. 

\begin{Def} \label{sm-eq-10} Let $(A, \wedge, \thicksim, \backsim, 1)$ be a pseudo equality algebra. 
A \emph{state-morphism operator} on $A$ is a map $\sigma:A\longrightarrow A$ satisfying the following conditions for 
all $x, y\in A$: \\
$(SM_1)$ $\sigma(x\thicksim y)=\sigma(x)\thicksim \sigma(y);$ \\
$(SM_2)$ $\sigma(x\backsim y)=\sigma(x)\backsim \sigma(y);$ \\
$(SM_3)$ $\sigma(x\wedge y)=\sigma(x)\wedge \sigma(y);$ \\
$(SM_4)$ $\sigma(\sigma(x))=\sigma(x)$. \\
The pair $(A, \sigma)$ is called a \emph{state-morphism pseudo equality algebra}.
\end{Def}

Denote $\mathcal{SM}_{EQA}(A)$ the set of all state-morphisms on a pseudo equality algebra $A$. 

\begin{prop} \label{sm-eq-20} A state-morphism operator is an order-preserving homomorphism.
\end{prop}
\begin{proof} Let $(A, \wedge, \thicksim, \backsim, 1)$ be a pseudo equality algebra and let $x\in A$. \\
According to $(SM_1)$ we have $\sigma(1)=\sigma(x\thicksim x)=\sigma(x)\thicksim \sigma(x)=1$. 
Taking into consideration conditions $(SM_1)-(SM_3)$, it follows that $\sigma$ is a homomorphism on $A$. \\
Consider $x, y\in A$ such that $x\le y$, that is $x\wedge y=x$. 
Applying $(SM_3)$ we get: $\sigma(x)=\sigma(x\wedge y)=\sigma(x)\wedge \sigma(y)\le \sigma(y)$, thus $\sigma$ is an 
order-preserving homomorphism on $A$.
\end{proof}

\begin{ex} \label{sm-eq-30} Let $(A, \wedge, \thicksim, \backsim, 1)$ be a pseudo equality algebra. 
Then the maps ${\bf 1}_A, \Id_A:A\longrightarrow A$, defined by ${\bf 1}_A(x)=1$ and $Id_A(x)=x$ for all $x\in A$ 
are state-morphisms on $A$. 
\end{ex}

\begin{ex}\label{sm-eq-30-10} Let $A_1$ and $A_2$ be two pseudo equality algebras and let $A$ be the pseudo equality algebra defined in Example \ref{ps-ex-70}. 
Then the maps $\sigma_1, \sigma_2:A\longrightarrow A$ defined by $\sigma_1(x,y)=(x,x)$ and $\sigma_2(x,y)=(y,y)$ for all for all $(x,y)\in A$ are state-morphisms on $A$.
\end{ex}

\begin{ex} \label{sm-eq-30-20} Let $(A, \wedge, \thicksim, \backsim, a, 1)$ be an $a$-compatible 
pseudo equality algebra. Then the map $\sigma_a:A\longrightarrow A$ defined by  
$\sigma_a(x)=x^{\thicksim_a\backsim_a}$ for all $x\in A$, is a state-morphism on $A$. \\
Obviously $\sigma_1=\Id_A$.
\end{ex}

\begin{prop} \label{sm-eq-40} Let $(A, \sigma)$ be a state-morphism pseudo equality algebra. 
Then the following hold:\\
$(1)$ $\Ker(\sigma)\in {\mathcal DS}_n(A);$ \\
$(2)$ $\Ker(\sigma)=\{x\thicksim \sigma(x) \mid x\in A\}=\{\sigma(x)\backsim x \mid x\in A\};$ \\
$(3)$ if $A$ is extensive and $\Ker(\sigma)=\{1\}$, then $\sigma=\Id_A;$ \\
$(4)$ if $A$ is extensive and simple, then $\mathcal{SM}_{EQA}(A)=\{{\bf 1}_A, \Id_A\}$.
\end{prop}
\begin{proof} 
$(1)$ Obviously $1\in \Ker(\sigma)$. \\
Consider $x, y\in A$ such that $x\in \Ker(\sigma)$ and $x\le y$. \\
Then by Proposition \ref{sm-eq-20}, $1=\sigma(x)\le \sigma(y)$, so $\sigma(y)=1$. Hence $y\in \Ker(\sigma)$. \\
Take $x, y\in A$ such that $x, y\thicksim x \in \Ker(\sigma)$, that is 
$\sigma(x)=\sigma(y\thicksim x)=1$. \\ 
We have: $\sigma(y)=\sigma(y)\thicksim 1=\sigma(y)\thicksim \sigma(x)=\sigma(y\thicksim x)=1$, that is 
$y\in \Ker(\sigma)$. \\
We conclude that $\Ker(\sigma)$ is a deductive system of $A$. \\
Consider $x, y\in A$ such that $x\thicksim y, y\thicksim x \in \Ker(\sigma)$, that is 
$\sigma(x\thicksim y)=\sigma(y\thicksim x)=1$. 
It follows that $\sigma(x)\thicksim \sigma(y)=1$ and $\sigma(y)\thicksim \sigma(x)=1$.  
Hence, by Proposition \ref{ps-eq-10}$(3)$, $\sigma(x)=\sigma(y)$. 
We get $\sigma(x\backsim y)=\sigma(x)\backsim \sigma(y)=1$ and 
$\sigma(y\backsim y)=\sigma(y)\backsim \sigma(x)=1$, hence $y\backsim y, x\backsim y \in \Ker(\sigma)$. \\ 
Similarly from $y\backsim x, x\backsim y \in \Ker(\sigma)$ we get $x\thicksim y, y\thicksim x \in \Ker(\sigma)$. \\ 
Thus $\Ker(\sigma)\in {\mathcal DS}_n(A)$. \\
$(2)$ Denote $X=\{x\thicksim \sigma(x) \mid x\in A\}$. Suppose $x\in \Ker(\sigma)$, that is $\sigma(x)=1$. \\
It follows that $x=x\thicksim 1=x\thicksim \sigma(x) \in X$, hence $\Ker(\sigma)\subseteq X$. \\
Conversely, consider $y\in X$, thus there exists $x\in A$ such that $y=x\thicksim \sigma(x)$. \\
We have $\sigma(y)=\sigma(x\thicksim \sigma(x))=\sigma(x)\thicksim \sigma(\sigma(x))=\sigma(x)\thicksim \sigma(x)=1$, 
thus $y\in \Ker(\sigma)$. \\
Hence $X\subseteq \Ker(\sigma)$, and we conclude that $\Ker(\sigma)=\{x\thicksim \sigma(x) \mid x\in A\}$. \\
Similarly $\Ker(\sigma)=\{\sigma(x)\backsim x \mid x\in A\}$. \\
$(3)$ By $(2)$, $x\thicksim \sigma(x), \sigma(x)\backsim x \in \Ker(\sigma)$ for all $x\in A$, hence 
$x\thicksim \sigma(x)=\sigma(x)\backsim x=1$, that is $\sigma(x)\le x$ for all $x\in A$. 
Since $A$ is extensive, it follows that $\sigma=\Id_A$. \\ 
$(4)$ Since $A$ is a simple pseudo equality algebra, we have $\Ker(\sigma)=\{1\}$ or $\Ker(\sigma)=A$. 
Applying $(3)$, it follows that $\sigma=\Id_A$ or $\sigma={\bf 1}_A$, that is 
$\mathcal{SM}_{EQA}(A)=\{{\bf 1}_A, \Id_A\}$.
\end{proof}

\begin{lemma} \label{sm-eq-40-10} Let $(A, \wedge, \thicksim, \backsim, a, 1)$ be a pointed pseudo equality algebra 
and $\sigma:A\longrightarrow A$ be a state-morphism operator on $A$ such that $\sigma(a)=a$. Then: \\
$(1)$ $\sigma(x^{\thicksim_a})=\sigma(x)^{\thicksim_a}$ and $\sigma(x^{\backsim_a})=\sigma(x)^{\backsim_a};$ \\
$(2)$ $\sigma(x^{\thicksim_a\backsim_a})=\sigma(x)^{\thicksim_a\backsim_a}$ and 
$\sigma(x^{\backsim_a\thicksim_a})=\sigma(x)^{\backsim_a\thicksim_a}$.
\end{lemma}
\begin{proof} Applying $(SM_1)$ and $(SM_2)$ we have:\\
$(1)$ $\sigma(x^{\thicksim_a})=\sigma(a\thicksim x)=\sigma(a)\thicksim \sigma(x)=
a\thicksim \sigma(x)=\sigma(x)^{\thicksim_a}$ and similarly $\sigma(x^{\backsim_a})=\sigma(x)^{\backsim_a}$. \\
$(2)$ It is a consequence of $(1)$.
\end{proof}

\begin{theo} \label{sm-eq-50} For any pseudo equality algebra $(A, \wedge, \thicksim, \backsim, 1)$,  
$\mathcal{SM}_{EQA}(A)\subseteq \mathcal{IS}_{EQA}^{(I)}(A)$. 
\end{theo}
\begin{proof} Let $\sigma\in \mathcal{SM}_{EQA}(A)$. Obviously $(IS_1)$ is verified by Proposition \ref{sm-eq-20}, 
while $(IS_3)$ and $(IS_4)$ follow from $(SM_3)$ and $(SM_4)$. 
In order to prove $(IS_2)$, we apply Proposition \ref{ps-eq-30}: \\ 
$\sigma(y)\thicksim \sigma((x\wedge y\thicksim x) \backsim y)=
\sigma(y\thicksim ((x\wedge y\thicksim x) \backsim y))=\sigma(x\wedge y\thicksim x)$ and \\
$\sigma(y\thicksim (x\backsim x\wedge y)) \backsim \sigma(y)=
\sigma((y\thicksim (x\backsim x\wedge y)) \backsim y)=\sigma(x\backsim x\wedge y)$, \\
thus $(IS_2)$ is verified. 
We conclude that $\sigma\in \mathcal{IS}_{EQA}^{(I)}(A)$, that is 
$\mathcal{SM}_{EQA}(A)\subseteq \mathcal{IS}_{EQA}^{(I)}(A)$.
\end{proof}

\begin{rem} \label{sm-eq-50-10}
If $(A,\wedge,\thicksim,\backsim,1)$ is a commutative pseudo equality algebra, then according to 
Remark \ref{is-eq-40}, 
$\mathcal{SM}_{EQA}(A)\subseteq \mathcal{IS}_{EQA}^{(I)}(A)=\mathcal{IS}_{EQA}^{(II)}(A)$.
\end{rem}

\begin{ex} \label{sm-eq-50-20} In the Example \ref{is-eq-50}, \\
$\hspace*{2cm}$
$\mathcal{SM}_{EQA}(B)=\mathcal{IS}_{EQA}^{(I)}(B)=\mathcal{IS}_{EQA}^{(II)}(B)= 
\{\sigma_1, \sigma_2, \sigma_3, \sigma_4, \sigma_5, \sigma_6\}$. \\
This is in accordance with Remark \ref{sm-eq-50-10} and Example \ref{ps-ex-100}.
\end{ex}

\begin{theo} \label{sm-eq-60} Let $(A, \wedge, \thicksim, \backsim, a, 1)$ be a pointed $a$-compatible 
pseudo equality algebra and $\sigma:\mbox{\Reg}_a(A) \rightarrow \mbox{\Reg}_a(A)$ be a state-morphism operator on $\mbox{\Reg}_a(A)$ such that $\sigma(a)=a$. Then the mapping $\tilde \sigma:A \rightarrow A$ defined by $\tilde \sigma(x)=\sigma(x^{\thicksim_a\backsim_a})$ is a state-morphism operator on $A$ such that 
$\tilde \sigma_{\mid \mbox{\Reg}_a(A)}=\sigma$.
\end{theo}
\begin{proof} Since $\sigma$ satisfies $(SM_1)$, $(SM_2)$, applying $(C_1)$ we get:\\
$\hspace*{1cm}$
$\tilde \sigma (x\thicksim y)=\sigma((x\thicksim y)^{\thicksim_a\backsim_a})=
\sigma(x^{\thicksim_a\backsim_a}\thicksim y^{\thicksim_a\backsim_a})$ \\
$\hspace*{2.5cm}$
$=\sigma(x^{\thicksim_a\backsim_a})\thicksim \sigma(y^{\thicksim_a\backsim_a})=
\tilde \sigma (x)\thicksim \tilde \sigma (y)$. \\
Thus $\tilde \sigma$ satisfies $(SM_1)$. 
We can prove similarly that $\tilde \sigma$ satisfies $(SM_2)$. \\
Applying $(C_3)$ and $(SM_3)$ for $\sigma$ we have:\\
$\hspace*{1cm}$
$\tilde \sigma (x\wedge y)=\sigma((x\wedge y)^{\thicksim_a\backsim_a})=\sigma(x\wedge y)^{\thicksim_a\backsim_a}=
(\sigma(x)\wedge \sigma(y))^{\thicksim_a\backsim_a}$ \\
$\hspace*{2.5cm}$
$=\sigma(x)^{\thicksim_a\backsim_a}\wedge \sigma(y)^{\thicksim_a\backsim_a}=
\sigma(x^{\thicksim_a\backsim_a})\wedge \sigma(y^{\thicksim_a\backsim_a})=
\tilde \sigma (x)\wedge \tilde \sigma(y)$. \\
Hence $\tilde \sigma$ satisfies $(SM_3)$. 
Applying $(C_4)$ we have: \\
$\hspace*{1cm}$
$\tilde \sigma(\tilde \sigma(x))=\tilde \sigma(\sigma(x^{\thicksim_a\backsim_a}))=
\sigma(x^{\thicksim_a\backsim_a})^{\thicksim_a\backsim_a}=\sigma(x^{\thicksim_a\backsim_a\thicksim_a\backsim_a})=
\sigma(x^{\thicksim_a\backsim_a})=\tilde \sigma(x)$, \\
that is, $(SM_4)$. 
We conclude that $\tilde \sigma$ is a state-morphism operator on $A$. \\
If $x\in \Reg_a(A)$, then $\tilde \sigma(x)=\sigma(x^{\thicksim_a\backsim_a})=\sigma(x)$, so that 
$\tilde \sigma_{\mid \mbox{\Reg}_a(A)}=\sigma$.
\end{proof}

\begin{prop} \label{sm-eq-60-10} Let $(A, \wedge, \thicksim, \backsim, a, 1)$ be a pointed pseudo equality algebra 
$(a\ne 1)$ and let $s\in \mathcal{BS}^{(a)}_{EQA}(A)$ and $\sigma\in \mathcal{SM}_{EQA}(A)$ such that $\sigma(a)=a$.   
If $s_{\sigma}:A\longrightarrow A$, $s_{\sigma}(x)=s(\sigma(x))$ for all $x\in A$, then 
$s_{\sigma}\in \mathcal{BS}^{(a)}_{EQA}(A)$. 
\end{prop}
\begin{proof}
Let $s\in \mathcal{BS}^{(a)}_{EQA}(A)$ and $x, y\in A$. Then we have:\\
$\hspace*{1cm}$ 
$s_{\sigma}(x)+s_{\sigma}(x\wedge y\thicksim x)=s(\sigma(x))+s(\sigma(x\wedge y\thicksim x))$ \\
$\hspace*{4.7cm}$
$=s(\sigma(x))+s(\sigma(x)\wedge \sigma(y)\thicksim \sigma(x))$ \\
$\hspace*{4.7cm}$
$=s(\sigma(y))+s(\sigma(x)\wedge \sigma(y)\thicksim \sigma(y))$ \\
$\hspace*{4.7cm}$
$=s(\sigma(y))+s(\sigma(x\wedge y\thicksim y))=s_{\sigma}(y)+s_{\sigma}(x\wedge y\thicksim y)$. \\
Similarly $s_{\sigma}(x)+s_{\sigma}(x\backsim x\wedge y)=s_{\sigma}(y)+s_{\sigma}(y\backsim x\wedge y)$. \\
Moreover, $s_{\sigma}(1)=s(\sigma(1))=s(1)=1$ and $s_{\sigma}(a)=s(\sigma(a))=s(a)=0$. \\
It follows that $s_{\sigma}$ satisfies axioms $(BS_1)$, $(BS_2)$, $(BS_3)$, that is 
$s_{\sigma}\in \mathcal{BS}^{(a)}_{EQA}(A)$. 
\end{proof}

In what follows we recall the notion of a state-morphism on pseudo BCK-meet-semilattices and we investigate 
the connection between state-morphisms on a pseudo equality algebra $(A, \wedge, \thicksim, \backsim, 1)$ 
and state-morphisms on its corresponding pseudo BCK(pC)-meet-semilattice 
$\Psi(A)=(A, \wedge, \rightarrow, \rightsquigarrow, 1)$.

\begin{Def} \label{sm-eq-70} $\rm($\cite{Ciu7}$\rm)$ Let $(B,\wedge,\rightarrow,\rightsquigarrow,1)$ be a pseudo BCK-meet-semilattice. A homomorphism $\mu:B\longrightarrow B$ is called a \emph{state-morphism operator} on $B$ if $\mu^2=\mu$, where $\mu^2=\mu\circ \mu$. The pair $(B, \mu)$ is called a \emph{state-morphism pseudo BCK-meet semilattice}.
\end{Def}

Denote $\mathcal{SM}_{BCK}(B)$ the set of all state-morphisms on a pseudo BCK-meet-semilattice $B$. 

\begin{rem} \label{sm-eq-70-10} A state-morphism operator on a pseudo BCK-meet-semilattice $B$ is order-preserving. 
Indeed, let $x, y\in B$ such that $x\le y$. It follows that $\mu(x)\rightarrow \mu(y)=\mu(x\rightarrow y)=\mu(1)=1$, 
that is $\mu(x)\le \mu(y)$.
\end{rem}

\begin{ex} \label{sm-eq-80} If $(B,\wedge,\rightarrow,\rightsquigarrow,1)$ is a pseudo BCK-meet-semilattice,  
then the maps ${\bf 1}_B, \Id_B:B\longrightarrow B$, defined by ${\bf 1}_B(x)=1$ and $\Id_B(x)=x$ for all $x\in B$ 
are  state-morphism operators on $B$. 
\end{ex}

\begin{prop} \label{sm-eq-90} Let $(A, \wedge, \thicksim, \backsim, 1)$ be a pseudo equality algebra and 
$\mathcal{SM}_{BCK}(\Psi(A))$ be its corresponding pseudo BCK(pC)-meet-semilattice.  
Then $\mathcal{SM}_{EQA}(A)\subseteq \mathcal{SM}_{BCK}(\Psi(A))$. 
\end{prop}
\begin{proof}
Let $\sigma \in \mathcal{SM}_{EQA}(A)$. Since $\sigma$ is a homomorphism on $A$, we have 
$\sigma(x\rightarrow y)=\sigma(x\wedge y\thicksim x)=\sigma(x\wedge y)\thicksim \sigma(x)=
\sigma(x)\wedge \sigma(y)\thicksim \sigma(x)=
\sigma(x)\rightarrow \sigma(y)$. \\
Similarly $\sigma(x\rightsquigarrow y)=\sigma(x)\rightsquigarrow \sigma(y)$. \\
The other conditions are straightforward, thus $\sigma \in \mathcal{SM}_{BCK}(\Psi(A))$. \\
We conclude that $\mathcal{SM}_{EQA}(A)\subseteq \mathcal{SM}_{BCK}(\Psi(A))$.
\end{proof}

\begin{theo} \label{sm-eq-100} Let $(A, \wedge, \thicksim, \backsim, 1)$ be a linearly ordered symmetric pseudo  equality algebra and let $\Psi(A)=(A,\wedge,\rightarrow,\rightsquigarrow,1)$ be its corresponding pseudo BCK(pC)-meet-semilattice. \\ 
Then $\mathcal{SM}_{BCK}(\Psi(A))\subseteq \mathcal{SM}_{EQA}(A)$.
\end{theo}
\begin{proof} Let $\mu \in \mathcal{SM}_{BCK}(\Psi(A))$. 
Obviously $\mu$ satisfies conditions $(SM_3)$ and $(SM_4)$. \\
Consider $x, y\in A$ such that $x\le y$, so $\mu(x)\le \mu(y)$. Then: \\
$\hspace*{1cm}$
$\mu(x\thicksim y)=\mu(x\wedge y\thicksim y)=\mu(y\rightarrow x)=\mu(y)\rightarrow \mu(x)$ \\
$\hspace*{2.6cm}$
$=\mu(x)\wedge \mu(y)\rightarrow \mu(x)=\mu(x)\thicksim \mu(y)$. \\
$\hspace*{1cm}$
$\mu(x\backsim y)=\mu(y\thicksim x)=\mu(y)\thicksim \mu(x)=\mu(x)\backsim \mu(y)$. \\
Suppose $y\le x$, thus $\mu(y)\le \mu(x)$. We get: \\
$\hspace*{1cm}$
$\mu(x\backsim y)=\mu(x\backsim x\wedge y)=\mu(x\rightsquigarrow y)=\mu(x)\rightsquigarrow \mu(y)$ \\
$\hspace*{2.6cm}$
$=\mu(x)\rightsquigarrow \mu(x)\wedge \mu(y)=\mu(x)\backsim \mu(y)$. \\
$\hspace*{1cm}$
$\mu(x\thicksim y)=\mu(y\backsim x)=\mu(y)\backsim \mu(x)=\mu(x)\thicksim \mu(y)$. \\
Thus $\mu$ satisfies $(SM_1)$ and $(SM_2)$.\\
We conclude that $\mu \in \mathcal{SM}_{EQA}(A)$, that is $\mathcal{SM}_{BCK}(\Psi(A))\subseteq \mathcal{SM}_{EQA}(A)$.
\end{proof}

\begin{theo} \label{sm-eq-110} Let $\mathcal{B}=(B, \wedge, \rightarrow, \rightsquigarrow, 1)$ be a 
pseudo BCK(pC)-meet-semilattice and let 
$\Phi(B)=(B, \wedge, \thicksim=\leftarrow, \backsim=\rightsquigarrow, 1)$ be its corresponding 
pseudo equality algebra. \\
Then $\mathcal{SM}_{BCK}(B)\subseteq \mathcal{SM}_{EQA}(\Phi(B))$.
\end{theo}
\begin{proof} Consider $\mu \in \mathcal{SM}_{BCK}(B)$ and let $x, y\in A$. We have: \\
$\hspace*{2.6cm}$
$\mu(x\thicksim y)=\mu(y\rightarrow x)=\mu(y)\rightarrow \mu(x)=\mu(x)\thicksim \mu(y)$ and \\
$\hspace*{2.6cm}$
$\mu(x\backsim y)=\mu(x\rightsquigarrow y)=\mu(x)\rightsquigarrow \mu(y)=\mu(x)\backsim \mu(y)$, \\
hence $(SM_1)$ and $(SM_2)$ are verified. \\
Since $\mu \in \mathcal{SM}_{BCK}(B)$, then $(SM_3)$ and $(SM_4)$ are also satisfied. \\
It follows that $\mu \in \mathcal{SM}_{EQA}(\Phi(\mathcal{B}))$, hence 
$\mathcal{SM}_{BCK}(B)\subseteq \mathcal{SM}_{EQA}(\Phi(B))$.
\end{proof}

\begin{ex} \label{sm-eq-120} Consider the BCK(C)-lattice $(B,\wedge,\rightarrow,1)$ and its corresponding pseudo
equality algebra $\Phi(B)=(B, \wedge, \thicksim, \backsim, 1)$ from Example \ref{ps-ex-50}.  
Let $\mu_i:B\longrightarrow B$, $i=1,\cdots,6$ be the maps defined in Example \ref{is-eq-150}. 
Then $\mathcal{SM}_{BCK}(B)=\mathcal{SM}_{EQA}(\Phi(B))=\{\mu_1, \mu_2, \mu_3, \mu_4, \mu_5, \mu_6\}$. 
\end{ex}

$\vspace*{5mm}$

\section{Concluding remarks}

As mentioned in the Introduction, a new concept of FTT has been developed having the structure of truth values 
formed by a linearly ordered good EQ$_{\Delta}$-algebra (\cite{Nov9}) and a fuzzy-equality based logic 
called EQ-logic has also been introduced (\cite{Nov10}).
The study of pseudo equality algebras is motivated by the goal to develop appropriate algebraic semantics 
for FTT, so a concept of FTT should be introduced based on these algebras.  
At the same time, pseudo equality algebras could be intensively studied from an algebraic point of view. 
In this paper we introduced and studied the internal states and the state-morphism operators on 
pseudo equality algebras, and we proved new results regarding these structures. 
Since the above topics are of current interest we suggest further directions of research: \\
$-$ Characterize deductive systems generated by a subset of a pseudo equality algebra in terms of operations $\thicksim$, $\backsim$ and $\wedge$. \\
$-$ Define the notion of state-deductive system and investigate the correspondence between the existence of internal states and the maximal and normal state-deductive systems. \\
$-$ Define and characterize subdirectly irreducible state pseudo equality algebras. \\
$-$ Develop a pseudo equality logic. \\
$-$ Develop a fuzzy type theory based on pseudo equality algebras.




$\vspace*{5mm}$

\setlength{\parindent}{0pt}

\vspace*{3mm}

\begin{flushright}
\begin{minipage}{148mm}\sc\footnotesize
Lavinia Corina Ciungu\\
Department of Mathematics \\
University of Iowa \\
14 MacLean Hall, Iowa City, Iowa 52242-1419, USA \\
{\it E--mail address}: {\tt lavinia-ciungu@uiowa.edu}

\end{minipage}
\end{flushright}

\end{document}